\documentclass[a4paper,10pt]{article}
\usepackage[utf8]{inputenc}
\usepackage[T1]{fontenc}
\usepackage[french,english]{babel}
\usepackage{amssymb}
\usepackage{enumerate}
\usepackage{amsthm}
\usepackage{amsmath}
\usepackage{longtable}
\usepackage{caption}
\usepackage{relsize}
\usepackage{xcolor}
\usepackage{enumitem}
\usepackage[document]{ragged2e}
\usepackage{hyperref}
\usepackage{geometry}
\usepackage{bm}
\usepackage{bbm}
\usepackage[capitalize]{cleveref}
\usepackage{thmtools}
\usepackage{cases}
\usepackage{libertine}

\usepackage[title]{appendix}
\usepackage[backend=bibtex,
            giveninits=true,
            style=numeric,
            doi=false,
            isbn=false,
            url=false,
            eprint=false,
            maxbibnames=99
            ]{biblatex}
\usepackage[title]{appendix}

\hypersetup{
    colorlinks=true,
    linkcolor=[rgb]{0.2, 0, 0.7},
    filecolor=magenta,
    urlcolor=cyan,
    bookmarks=true,
    citecolor=[rgb]{0,0.4,0},
}

\crefformat{equation}{#2\normalfont{{(\textup{#1})}}#3}

\makeatletter
\newcommand{\subjclass}[2][1991]{%
  \let\@oldtitle\@title%
  \gdef\@title{\@oldtitle\footnotetext{#1 \emph{Mathematics subject classification.} #2}}%
}
\newcommand{\keywords}[1]{%
  \let\@@oldtitle\@title%
  \gdef\@title{\@@oldtitle\footnotetext{\emph{Keywords and phrases.} #1.}}%
}
\makeatother

\newcommand{\RR}{\mathbb{R}}
\newcommand{\NN}{\mathbb{N}}
\newcommand{\TT}{\mathbb{T}}
\newcommand{\EE}{\mathbb{E}}
\newcommand{\PP}{\mathbb{P}}

\newcommand{\cP}{\mathcal{P}}

\newcommand{\cO}{\mathcal{O}}
\newcommand{\cL}{\mathcal{L}}

\newcommand{\norm}[2]{ \left \lVert {#1}  \right \rVert_{#2}}
\newcommand{\module}[1]{\left \lvert {#1} \right \rvert}

\DeclareMathOperator{\loc}{loc}

\DeclareMathOperator{\diver}{div}

\theoremstyle{plain}
\newtheorem{thm}{Theorem}[section]
\newtheorem{cor}[thm]{Corollary}
\newtheorem{lem}[thm]{Lemma}
\newtheorem{prop}[thm]{Proposition}

\theoremstyle{definition}
\newtheorem{defi}[thm]{Definition}

\theoremstyle{remark}
\newtheorem{rem}[thm]{Remark}
\newtheorem{ex}[thm]{Example}

\crefname{thm}{Theorem}{Theorems}
\crefname{cor}{Corollary}{Corollaries}
\crefname{lem}{Lemma}{Lemmata}
\crefname{prop}{Proposition}{Propositions}
\crefname{def}{Definition}{Definitions}
\crefname{rem}{Remark}{Remarks}
\crefname{ex}{Example}{Examples}
\usepackage[pagewise]{lineno}

\title{Approximation and perturbations of stable solutions to a stationary mean field game system}
\author{Jules Berry \\ {\small Univ Rennes, INSA, CNRS, IRMAR - UMR 6625, Rennes F-35000, France} \\ {\small jules.berry@insa-rennes.fr} \and Olivier Ley \\ {\small Univ Rennes, INSA, CNRS, IRMAR - UMR 6625, Rennes F-35000, France} \\ {\small olivier.ley@insa-rennes.fr} \and Francisco J. Silva \\ {\small Univ Limoges,  Faculté des Sciences et Techniques, XLIM-DMI, UMR-CNRS 7252, 87060 Limoges, France} \\ {\small francisco.silva@unilim.fr}}
\date{\today}

\begin{document}
\subjclass[2020]{35J47, 35Q89, 65N20}
\keywords{Mean field games, Newton's method, finite element method, numerical methods}

\maketitle

\justify

\begin{abstract}
   This work introduces a new general approach for the numerical analysis of stable equilibria to second order  mean field games systems in cases where the uniqueness of solutions may fail. We focus on a stationary case with a purely quadratic Hamiltonian. We propose an abstract framework to study these solutions by reformulating the mean field game system as an abstract equation in a Banach space. In this context, stable equilibria turn out to be regular solutions to this equation, meaning that the linearized system is well-posed. We provide three applications of this property: we study the sensitivity analysis of stable solutions, establish error estimates for their finite element approximations, and prove the local converge of Newton's method in infinite dimensions.
\end{abstract}

\section{Introduction}

Mean Field Games (MFG for short) were introduced independently by Lasry-Lions \cite{LL2007,LL2006,LL2006a} and Huang-Caines-Malhamé \cite{HMC2006,HCM2007}. The goal of this theory is to study (stochastic) differential games with a large number of interchangeable players. We refer the reader to \cite{ACDPS2020,CD2018,BFY2013,GS2014,L2007} for general references on this topic.

The numerical analysis of MFG systems introduced in \cite{LL2007,LL2006,LL2006a} has been extensively studied under a monotonicity assumption also introduced by Lasry and Lions,  see \cite[Chapter 4]{ACDPS2020}, \cite{A2013}, and the references therein. Indeed, the latter provides a sufficient condition for the uniqueness of solutions to MFG systems, which allows to show the convergence of numerical methods in \cite{ACC2013} and error estimates in \cite{BLP2023}. In the absence of this monotonicity assumption, uniqueness may fail (see \cite{BC2018,BF2019,CT2019,C2019}) and the study of MFG systems with several solutions is delicate both from the theoretical and numerical points of view. In \cite{BC2018}, Briani and Cardaliaguet have defined a particular notion of solution for second order potential MFG systems, the so-called {\it stable solutions} (see also \cite{BN2023} for a related notion in the context of first order mean field games). These solutions may not be unique but the authors show in \cite{BC2018} that they have some interesting properties motivating their name: stable solutions are isolated and the fictitious play algorithm, introduced in Cardaliaguet-Hadikhanloo \cite{CH2017}, converges locally to these solutions. We also mention the recent work by Tang-Song \cite{TS2024}, where the authors implement a smoothed policy iteration method to locally approximate stable solutions. In this paper, we provide new results in this direction, which reinforce the importance of the notion of stable solutions. We are going to prove that stable solutions are indeed stable under perturbations and that local convergence holds for their approximations by finite element methods and Newton iterations.\\

This paper is the first in a series of works dealing with the numerical analysis of stable equilibria to MFG models. Our goal is to introduce a general framework that covers different types of MFG systems under fairly general assumptions. In order to convey our main ideas, we focus in this paper on the following stationary MFG system
\begin{equation}
\label{eq:mfg}
 \begin{cases}
    - \Delta u + \frac{1}{2} \module{Du}^2 + \lambda u = f(m) \quad & \textnormal{in } \TT^d, \\
    - \Delta m - \diver \left ( m Du \right ) + \lambda m = \lambda m_0 \quad & \textnormal{in } \TT^d,
 \end{cases}
\end{equation}
where $\lambda > 0$ is a given constant and $m_0 \colon \TT^d \to \RR$ and $f \colon \RR \to \RR$ are given functions. This system has been introduced in the monograph by Bensoussan-Frehse-Yam \cite[Chapter 7]{BFY2013} and has been furtherly studied in \cite{GF2018,GMT2020}. We briefly recall its interpretation in \cref{section:MFG}. We have made the choice to consider this quite restrictive setting in order to simplify the notations and draw attention to the main ideas. However, we expect that our results can be extended to the case of Hamiltonians $H(x,p)$ that are twice continuously differentiable with respect to $p$, with bounded Hessian, and a coupling function $f(x,m)$.  When the Lasry-Lions monotonicity condition is used, we also require the (strict) convexity of the Hamiltonian.
\\

Let us now present the main contribution of this paper. We reformulate system \eqref{eq:mfg} in the form
\[
 F(u,m) = 0,
\]
where $F\colon X \to X$ is a nonlinear mapping, defined on a Banach space $X$, having the form
\[
 F= I + T \circ G.
\]
More precisely, given a suitable Banach space $Z$, we choose $T \colon Z \to X$ as the linear operator defined by $T(f,g) = (v,\rho)$, where $(v,\rho)$ solves
\[
 \begin{cases}
  - \Delta v + \lambda v = f & \quad \textnormal{in } \TT^d, \\
  - \Delta \rho + \lambda \rho = g  & \quad \textnormal{in } \TT^d,
 \end{cases}
\]
and $G \colon X \to Z$ the nonlinear mapping defined by
\[
 G(v,\rho) = \left ( \frac{1}{2} \module{Dv}^2 - f(\rho), - \diver \left (\rho Du \right ) - \lambda m_0 \right ).
\]
We refer the reader to \cref{section:Reformulation} for the details of this reformulation. In the case where the mapping $G$ is differentiable, the mapping $F$ is also differentiable  and the stability of a solution $(u,m)$ to \eqref{eq:mfg} is equivalent to the injectivity of the differential
\[
 dF[u,m] = I + T \circ dG[u,m].
\]
Thus, if $T$ is a compact operator and $(u,m)$ is a stable solution to \eqref{eq:mfg}, then $dF[u,m]$ is an injective perturbation of the identity by a compact linear operator. Henceforth, by the Fredholm alternative, we deduce that $dF[u,m]$ is an isomorphism on $X$. This isomorphism property will be rigorously established below for Banach spaces of the form
\[
 X = C^{2,\gamma}(\TT^d) \times C^{2,\beta}(\TT^d) \quad \textnormal{and} \quad X = W^{1,p}(\TT^d) \times L^q(\TT^d),
\]
but many other choices are possible depending on the application in mind.
\\

We now describe three applications of the above isomorphism property for stable solutions.

A first and straightforward application, which follows from the implicit function theorem, concerns the sensitivity analysis of stable solutions to \eqref{eq:mfg} under perturbations of the coupling function $f$ and the distribution $m_0$ (\cref{prop:perturbations} below).

In the second application, we make use of the Brezzi-Rappaz-Raviart theory on the approximation of nonlinear problems (see \cite{BRR1980,GR1986,CR1990,CR1997}) to obtain existence and error estimates for finite element approximations of stable solutions (\cref{thm:fem_approx} below). The finite element approximation of a MFG system similar to \eqref{eq:mfg} has been studied by Osborne and Smears in \cite{OS2024a} (see \cite{OS2023} for a parabolic MFG system), where the convergence of the approximations is established by using compactness arguments which do not provide error estimates.  While finishing this paper, we have learnt about the recent work \cite{OS2024b} by the same authors addressing this issue. Regarding the convergence analysis, the results in \cite{OS2024a,OS2024b} deal only with the case where the coupling term $f$ is (strongly) monotone.
In contrast, our results apply locally around any stable solution without requiring any monotonicity of the coupling and rely on a completely different approach.

In our last application, we provide convergence rates for the iterates of Newton's method in infinite dimension applied to \eqref{eq:mfg} in various functional spaces (\cref{thm:newton_1,thm:newton_2,thm:newton_3} below). Let us also mention that the analysis of Newton's method in infinite dimensions to approximate the solution to a time-dependent MFG system, with monotone couplings, has been  recently addressed by Camilli and Tang \cite{CT2024} using different techniques. Compared with their approach, our result follows directly from classical convergence results of Newton's iterates in function spaces (see~\cite{HPUU2009,Z1993,DR2014}) which allow us to deal with non-monotone couplings and to obtain convergence rates in stronger norms.
\\

The paper is structured as follows. In the preliminary \cref{section:MFG}, we study the well-posedness of \eqref{eq:mfg} and provide some useful estimates on the solutions. We then introduce the definition of stable solution to \eqref{eq:mfg}, which is similar to the one proposed in \cite{BC2018}, and give sufficient conditions for the existence of such solutions. In \cref{section:Reformulation} we reformulate solutions to \eqref{eq:mfg} as zeros of the nonlinear mapping $F$ and prove the isomorphism property of its differential under a continuous differentiability assumption. The latter is rigorously justified in \cref{section:differentiability} in the case of Hölder and Sobolev spaces. Finally, applications to the sensitivity analysis of \eqref{eq:mfg}, to its finite element approximation, and to the convergence of Newton's method, are studied in \cref{section:applications}.

\paragraph{Notations} For $k \in \NN$ and $\alpha \in (0,1]$, we write $C^{k,\alpha}(\TT^d)$ the usual Hölder space on $\TT^d$,  {\it i.e.}
\[
 C^{k,\alpha}(\TT^d) = \left \{ u \in C^k(\TT^d) : \norm{u}{C^{k,\alpha}} < +\infty \right \},
\]
where
\[
 \norm{u}{C^{k,\alpha}} = \sum_{\module{j} \leq k} \norm{\partial^j u}{L^\infty} + \sum_{\module{j} = k} \left [ \partial^j u \right]_{\alpha},
\]
with
\[
\left [ u \right]_{\alpha} = \sup_{\substack{y,x \in \TT^d \\ x \neq y}} \frac{\module{u(x) - u(y)}}{\module{x -y}^\alpha}.
\]

Similarly, $C^{k,\alpha}_{\loc}(\RR)$ is the space of locally Hölder continuous functions on $\RR$, \textit{i.e.} $f \in C^{k,\alpha}_{\loc}(\RR)$ if $f \in C^{k,\alpha}(\Omega)$ for every bounded open subset $\Omega \subset \RR$. We also write $C_b^{k}(\RR)$ the set of $k$ times continuously differentiable functions $f$ on $\RR$ such that the derivatives $f^{(j)}$, for $0 \leq j \leq k$, are bounded. We denote by $\cP(\TT^d)$ the space on probability measures over $\TT^d$ and we always identify a measure $m \in \cP(\TT^d)$ with its density, which we also denote $m$, provided that the latter exists. For $1 < p \leq \infty$ the dual of the Sobolev space $W^{1,p}(\TT^d)$ is denoted by $W^{-1,p'}(\TT^d)$, where $1/p + 1/{p'} = 1$, and we reserve the notation $H^{-1}(\TT^d)$ for the dual of $H^1(\TT^d)$.

For Banach spaces $X$ and $Y$ and a mapping $\Phi \colon X \to Y$, we write $d\Phi[x]$ the Fréchet differential of $\Phi$ at $x \in X$, when it exists. We also use the notation  $Y \hookrightarrow X$ when $Y$ is continuously embedded in $X$.

By convention, we do not specify integration domains when integrals are considered on $\TT^d$, \textit{i.e.},
\[
 \int f \ d x := \int_{\TT^d} f \ dx.
\]

\section{The mean field game system}
\label{section:MFG}

In this section we establish some properties of the MFG system \eqref{eq:mfg}. We first state existence of solutions to \eqref{eq:mfg} as well as a uniqueness result under the Lasry-Lions monotonicity condition on the coupling $f$. We then define stable solutions to \eqref{eq:mfg} following \cite{BC2018} and prove that they are isolated. Finally, we provide sufficient conditions to ensure that any classical solution to \eqref{eq:mfg} is stable.\\

Let us begin by describing the mean field game interpretation of system \eqref{eq:mfg}. We consider a typical player whose dynamics is governed by the following controlled stochastic differential equation
\[
 \begin{cases}
   dX_t^{\alpha} = \alpha(X_t^{\alpha}) dt + \sqrt 2 dB_t \quad \textnormal{for all } t >0, \\
   X_0^{\alpha} = x,
 \end{cases}
\]
where $B_t$ is a $d$-dimensional Brownian motion and $\alpha$ is a feedback control. Assume that the player forecasts $\hat \rho$ as being the evolution of the distribution of players, from which it is possible to compute the weighted averaged density $\hat m$ defined by
\begin{equation}\label{eq:averaged_density}
 \hat m(x) = \lambda \int_0^{+\infty} \hat \rho(t,x) e^{-\lambda t}\ dt
\end{equation}
for some given $\lambda > 0$. Then this player aims to minimize the following cost
\[
 J(x,\alpha) = \EE \left [ \int_0^{+\infty} \left (\frac{\module{\alpha(X_t^{\alpha})}^2}{2} + f(\hat m(X_t^\alpha)) \right) e^{- \lambda t}\ dt \right ].
\]
This yields the Hamilton-Jacobi equation for the value function $u(x) := \inf_{\alpha} J(x,\alpha)$ :
\[
 -\Delta u + \frac{\module{Du}^2}{2} + \lambda u = f(\hat m).
\]
Since all the players are assumed to be interchangeable, they should all play according to the optimal strategy provided by the Hamilton-Jacobi equation, \textit{i.e.} $\alpha^\star(x) = - Du(x)$ (see \cite{FR2012,FS2006,YZ1999}). The player must then update the forecasted density by solving the Fokker-Planck equation
\begin{equation}\label{eq:fokker-planck}
 \begin{cases}
 \partial_t \rho - \Delta \rho - \diver(\rho Du) = 0,\\
 \rho(0) = m_0,
 \end{cases}
\end{equation}
where $m_0$ is the initial distribution of agents. This yields an updated averaged density $m$ through \eqref{eq:averaged_density}. Using integration by parts in \eqref{eq:averaged_density} one can easily derive the following equation on $m$
\begin{equation}\label{eq:model_fokker_plack}
 - \Delta m - \diver(m Du) + \lambda m = \lambda m_0.
\end{equation}

As usual we define an MFG equilibrium as being a fix point of this procedure which corresponds to a solution to \eqref{eq:mfg}. Notice that once we have a solutions to \eqref{eq:mfg} it is possible to recover the probability density function $\rho$ by solving the Fokker-Planck equation \eqref{eq:fokker-planck}.\\

In all of this paper we make the assumption that
\begin{equation}\label{h:m0}
 m_0 \in \cP(\TT^d) \cap C^{0,\alpha}(\TT^d) \quad  \textnormal{for some } \alpha \in (0,1).
\end{equation}

\subsection{Well-posedness}

The following result was established in \cite[Chapter 7]{BFY2013} by approximation. For the sake of self-containedness we provide a proof in Appendix \ref{section:well_posed} based on the classical method introduced by Lasry-Lions.
\begin{thm}\label{thm:well_posed}
Assume that $f \in W^{1,\infty}(\RR)$. Then there exists a classical solution $(u,m) \in C^{2,\alpha}(\TT^d) \times C^{2,\alpha}(\TT^d)$ to \eqref{eq:mfg}, where the constant $\alpha$ is fixed in \eqref{h:m0}. Furthermore, if $f' \geq 0$ or if $\lambda$ is large enough, then this solution is unique.
\end{thm}

We now turn to technical results, which will be used throughout the paper. The first one recalls some properties of Fokker-Planck type equations.
\begin{lem}\label{lem:fokker_planck}
 Let $b \in L^\infty(\TT^d; \RR^d)$, $f\in L^2(\TT^d)$, and $g \in L^2(\TT^d;\RR^d)$. There exists a weak solution $m \in H^1(\TT^d)$ to
 \begin{equation}\label{eq:fokker_planck_general}
  - \Delta m - \diver(mb) + \lambda m = f + \diver(g) \quad \textnormal{in } \TT^d.
 \end{equation}
 Moreover,
 \begin{enumerate}[label={\rm(\roman*)}]
  \item  \label{item:fp_uniqueness} $m$ is the only element in $L^2(\TT^d)$ such that
  \[
  \int (-\Delta \varphi + b \cdot D\varphi + \lambda \varphi)m \ dx = \int f \varphi - g \cdot D\varphi\ dx \quad \textnormal{for every } \varphi \in H^2(\TT^d).
 \]

 \item \label{item:fp_positive} If $f \geq 0$ and $\diver(g) = 0$, then, either $f = 0$ and $m =0$, or $m > 0$.

 \item \label{item:fp_W1p}  If $p \geq 2$, $g \in L^p(\TT^d; \RR^d)$, and $f \in L^{q}(\TT^d)$, with
 \[
  \begin{cases}
   q = dp/(d+p) \quad & \textnormal{if } p \neq d \textnormal{ and } d \geq 2, \\
   q > d/2 \quad & \textnormal{if } p = d \textnormal{ and } d \geq 2, \\
   q = 1 \quad & \textnormal{if } d = 1,
  \end{cases}
 \]
 then $m \in W^{1,p}(\TT^d)$ and there exists a positive constant $C=C(\norm{b}{L^\infty},\lambda,d,p,q)$ such that  
  \begin{equation}\label{eq:W1p_estimate}
  \norm{m}{W^{1,p}} \leq C \left (\norm{m}{L^1} + \norm{g}{L^p} +\norm{f}{L^q} \right).
 \end{equation}

 \item \label{item:fp_DGNM} If $g \in L^p(\TT^d; \RR^d)$ and $f \in L^{q}(\TT^d)$ for $p>d$ and $q > d/2$, then there exists a positive constant $C=C(\norm{b}{L^\infty},\norm{f}{L^{p/2}},\norm{g}{L^p},\lambda,d,p,q)$ such that
 \begin{equation}\label{eq:m_bounded}
  \norm{m}{L^\infty} \leq C.
 \end{equation}
 \end{enumerate}
\end{lem}
\begin{proof}
Let us start by proving that
\begin{equation}
\label{eq:injectivite_op_div}
\left [ m\in  L^2(\TT^d) \textnormal{ and }
\int (-\Delta \varphi + b \cdot D\varphi + \lambda \varphi)m \ dx = 0 \quad \textnormal{for every } \varphi \in H^2(\TT^d) \right ] \;\Longrightarrow\;\; m=0.
\end{equation}
Let $\xi \in L^2(\TT^d)$ and $v \in H^2(\TT^d)$, be the unique solution to
 \begin{equation}\label{eq:fp_unique_dual}
  - \Delta v + b \cdot Dv + \lambda v = \xi \quad \textnormal{in } \TT^d.
 \end{equation}
Existence and uniqueness of such a solution is proved in the case of Dirichlet boundary conditions in \cite[Theorem 8.9]{GT2001}.
Using $m$ as a test function for \eqref{eq:fp_unique_dual}, we have
\[
 \int \xi m\ dx = \int( - \Delta v + b \cdot Dv + \lambda v)m \ dx= 0 .
\]
Since $\xi$ is arbitrary, we obtain $m=0$, which shows \eqref{eq:injectivite_op_div}. Existence of a weak solution $m \in H^1(\TT^d)$ to \eqref{eq:fokker_planck_general} for $\lambda > 0$ large enough is a standard consequence of the Lax-Milgram theorem. One can then extend this result to every $\lambda > 0$ by using Fredholm's alternative in conjunction with \eqref{eq:injectivite_op_div}, as in the proof of \cite[Theorem 8.3]{GT2001}. In particular, we have established \ref{item:fp_uniqueness}.

In addition, when $\xi \geq 0$, $\diver(g) = 0$ and $f \geq 0$, we obtain similarly that
\[
 \int \xi m\ dx = \int( - \Delta v + b \cdot Dv + \lambda v)m \ dx= \int f v \, dx \geq 0
\]
since, in this case, $v \geq 0$ as a consequence of the strong maximum principle \cite[Theorem 8.19]{GT2001}. It follows that $m \geq 0$. The strict positivity when $m \neq 0$ is then a consequence of Harnack's inequality \cite[Theorem 8.20]{GT2001}. This proves \ref{item:fp_positive}.

Finally, the $W^{1,p}$ estimate \ref{item:fp_W1p} is a direct consequence of \cite[Theorem 1.7.4]{BKRS2015} while \ref{item:fp_DGNM} follows from De Giorgi-Nash-Moser estimates \cite[Theorem 8.17]{GT2001}.
\end{proof}

\begin{rem}\label{rem:dual_sobolev}
 We recall from \cite[Theorem 10.41]{L2009} that, for $1 < r \leq \infty$, any element $h \in W^{-1, r'}(\TT^d)$, where $1/r+1/r' = 1$, can be identified with $g_1 + \diver (g_2)$ for $g_1 \in L^r(\TT^d)$, $g_2 \in L^r(\TT^d;\RR^d)$ and $\norm{h}{W^{-1,r'}} = \left ( \norm{g_1}{L^r}^r +  \norm{g_2}{L^r}^r \right )^{1/r}$. In particular the conclusions of \cref{lem:fokker_planck} can be extended to equations with right-hand side in $W^{-1, r'}(\TT^d)$ for appropriate values of $r$.
\end{rem}

The following proposition contains a priori estimates on classical solutions to \eqref{eq:mfg}.
\begin{prop}\label{prop:estimates_mfg}
 Assume $f \in C_b^0(\RR)$. Then there exists a positive constant $K = K(\norm{f}{L^\infty},d)$ such that for every classical solution $(u,m)$ to \eqref{eq:mfg} it holds that
 \begin{equation}\label{eq:u_bounded}
  \norm{u}{L^\infty} \leq \frac{\norm{f}{L^\infty}}{\lambda},
 \end{equation}
 \begin{equation}\label{eq:u_lipschitz}
  \norm{Du}{L^\infty} \leq K,
 \end{equation}
 and
 \begin{equation}\label{eq:laplacian_estimate}
  \norm{\Delta u}{L^\infty} \leq 2 \norm{f}{L^\infty} + \frac{K^2}{2} =: M.
 \end{equation}
 Furthermore, if $\lambda > M$, then
 \begin{equation}\label{eq:m_bounded_lamba}
  \norm{m}{L^\infty} \leq \frac{\lambda}{\lambda - M} \norm{m_0}{L^\infty}.
 \end{equation}
\end{prop}

\begin{proof}
Inequality \eqref{eq:u_bounded} is a direct consequence of the comparison principle for $u$ and \eqref{eq:u_lipschitz} is given by \cite[Theorem 1.1]{LN2017}, since any classical solution is also a continuous viscosity solution. The estimate \eqref{eq:laplacian_estimate} directly follows from the equation satisfied by $u$. For the last inequality \eqref{eq:m_bounded_lamba}, notice that we may rewrite the second equation in \eqref{eq:mfg} as
 \[
  -\Delta m - Du \cdot Dm + (\lambda - \Delta u)m = \lambda m_0 \quad \textnormal{in } \TT^d.
 \]
 Therefore when $\lambda > M$ we deduce from \eqref{eq:laplacian_estimate} that $\lambda - \Delta u > 0$ and therefore the Fokker-Planck equation satisfies a comparison principle from which the inequality follows.
\end{proof}

\subsection{Stable solutions}
\label{section:stable}
The following definition is taken from \cite{BC2018}.
\begin{defi}[Stable solutions]\label{def:stable}
 Let $f \in C^1_b(\RR)$ and let $(u,m)$ be a classical solution to \eqref{eq:mfg}. We say that $(u,m)$ is stable if $(v,\rho) = (0,0)$ is the unique classical solution to
 \begin{equation}\label{eq:mfg_lin}
 \begin{cases}
  - \Delta v + Du \cdot Dv + \lambda v = f'(m) \rho &  \quad \textnormal{in } \TT^d, \\
  - \Delta \rho - \diver \left (\rho Du \right) + \lambda \rho = \diver \left (m Dv \right )  &  \quad \textnormal{in } \TT^d.
 \end{cases}
\end{equation}
\end{defi}

\begin{defi}
   Let $f \in C^1_b(\RR)$ and let $(u,m)$ be a classical solution to \eqref{eq:mfg}. A pair $(v,\rho) \in H^1(\TT^d) \times L^2(\TT^d)$ is a weak solution to \eqref{eq:mfg_lin} if it satisfies
   \begin{equation}\label{eq:mfg_lin_weak_1}
    \int Dv \cdot D\varphi + Du \cdot Dv \varphi + \lambda v \varphi\ dx =  \int f'(m)\rho \varphi \ dx \quad \textnormal{for every } \varphi \in H^1(\TT^d)
   \end{equation}
   and
   \begin{equation}\label{eq:mfg_lin_weak_2}
    \int (-\Delta \psi + Du \cdot D \psi  + \lambda \psi ) \rho\ dx = - \int  m Dv \cdot D\psi \ dx \quad \textnormal{for every } \psi \in H^2(\TT^d).
   \end{equation}
\end{defi}

\begin{lem}\label{lem:mfg_lin_reg}
 Let $f \in C_b^{1}(\RR) \cap C^{1,1}_{\loc}(\RR)$ and $(u,m) \in C^{2,\alpha}(\TT^d) \times C^{2,\alpha}(\TT^d)$ be a classical solution to \eqref{eq:mfg} and $(v,\rho) \in H^1(\TT^d) \times L^2(\TT^d)$ be a weak solution to \eqref{eq:mfg_lin}. Then $(v,\rho) \in C^{2,\alpha}(\TT^d) \times C^{2,\alpha}(\TT^d)$ and is a classical solution to \eqref{eq:mfg_lin}.
\end{lem}

\begin{proof}
 Notice first that if $\rho \in C^{0,\alpha}(\TT^d)$, then by Schauder estimates \cite[Corollary 6.3]{GT2001} we have $v \in C^{2,\alpha}(\TT^d)$ and then also $\rho \in C^{2,\alpha}(\TT^d)$. It is therefore enough to prove that if $(v,\rho) \in H^1(\TT^d) \times L^2(\TT^d)$ is a weak solution to \eqref{eq:mfg_lin} then $\rho \in C^{0,\alpha}(\TT^d)$. For this, let $1 < p < \infty$, depending on $d$ and $\alpha$, be such that Morrey's inequality yields $W^{1,p}(\TT^d) \hookrightarrow C^{0,\alpha}(\TT^d)$. It is now enough to prove that $v \in W^{1,p}(\TT^d)$ which implies $\rho \in W^{1,p}(\TT^d)$ by \cref{lem:fokker_planck} \ref{item:fp_W1p}.

 Since $\rho \in L^2(\TT^d)$ and $f'$ is bounded we deduce from standard elliptic regularity that $v \in H^2(\TT^d)$. Therefore in the case $d = 1,2$ the fact that $v \in W^{1,p}(\TT^d)$ directly follows from Sobolev's inequality. In the rest of the proof we assume $d \geq 3$.

 Since $v \in H^2(\TT^d)$ we have that $\diver(mDv) = Dm \cdot Dv + m \Delta v$ is an element of $L^2(\TT^d)$. It is then well known that the second equation in \eqref{eq:mfg_lin} has a weak solution $\tilde \rho \in H^2(\TT^d)$. Using the property \ref{item:fp_uniqueness} in \cref{lem:fokker_planck}, we deduce that in fact $\tilde \rho = \rho$ and therefore $\rho \in H^2(\TT^d)$. In the case where $d = 3,4$ we have $\rho \in L^{r}(\TT^d)$ for every $2 \leq r < \infty$ and in particular $\rho \in L^{p}(\TT^d)$. Injecting this information in the equation satisfied by $v$ we conclude that $v \in W^{2,p}(\TT^d)$ (see \cite[Theorem 9.11]{GT2001}) and the conclusion follows. We may therefore assume that $d \geq 5$ and in this case Sobolev's inequality yields $\rho \in L^{q_1}(\TT^d)$, where $q_1 = \frac{2d}{d-4}$ and, arguing as above, that $v \in W^{2,q_1}(\TT^d)$.

 We now know that $\diver(mDv)$ belongs to $L^{q_1}(\TT^d)$. One may keep this bootstrap argument going and conclude that either we can obtain that $\rho \in W^{2,d/2}(\TT^d)$ after a finite number of steps, or there exists a sequence of real numbers $2 \leq q_n < d/2$, defined by
 \[
  \begin{cases}
   q_{n+1} = \frac{d q_n}{d - 2q_n}, \\
   q_0 = 2,
  \end{cases}
 \]
 and such that $\rho \in W^{2,q_n}(\TT^d)$ for every $n \geq 1$.

 We claim that the latter case cannot happen. Indeed, if it were the case, we notice that $\frac{q_{n+1}}{q_n} = \frac{d}{d - 2q_n} > 1$ so that the sequence is increasing. In particular $q_n \geq 2$ for every $n$ and therefore $\frac{q_{n+1}}{q_n} > \frac{d}{d-4}$. It follows that $\frac{d}{2} > q_n \geq 2 \left(\frac{d}{d-4} \right)^n$ which yields a contradiction.

 We must therefore have $\rho \in W^{2,d/2}(\TT^d)$ after a finite number of steps and hence, also, $\rho \in L^s(\TT^d)$ for every $s \in [1, \infty)$. Using elliptic regularity one more time we have $v \in W^{1,p}(\TT^d)$. This concludes the proof according to the discussion at the beginning of the argument.
\end{proof}

The following proposition states that stable solutions to \eqref{eq:mfg}, although not unique in general, are isolated. The result is a straightforward adaptation of \cite[Proposition 4.2]{BC2018} with weaker norms. We provide the proof in Appendix \ref{section:isolated}.

\begin{prop}[Stable solutions are isolated]\label{prop:isolated}
 Let $f \in C_b^{1}(\RR)$, let $(u,m)$ be a stable solution to \eqref{eq:mfg}. Then there exists $R > 0$ such that, if $(\tilde u, \tilde m) \neq (u,m)$ is another classical solution to \eqref{eq:mfg}, then
 \[
  \norm{u - \tilde u}{H^1} + \norm{m - \tilde m}{L^2} > R.
 \]
\end{prop}

The next result provides two sufficient conditions for the stability of solutions to \eqref{eq:mfg}.

\begin{thm}\label{thm:stability}
 Let $f \in C_b^{1}(\RR)$ and $(u,m)$ be a classical solution to \eqref{eq:mfg}. Then $(u,m)$ is stable provided that one of the following conditions holds:
\begin{enumerate}[label={\rm(\roman*)}]
\item{\rm(Monotonicity of the coupling)} $f' \geq 0$.
\item{\rm(Large discount factor)} $\lambda > \Lambda$, where
\begin{equation}\label{eq:Lambda}
 \Lambda :=\max \left \{2 M, \frac{K^2}{2} + \norm{m_0}{L^\infty} \norm{f'}{L^\infty} \right \}
\end{equation}
and the constants $M$ and $K$ is given in \cref{prop:estimates_mfg}.
\end{enumerate}

\end{thm}

\begin{proof}
  Let $(v,\rho)$ be a weak solution to \eqref{eq:mfg_lin}. Using \cref{lem:mfg_lin_reg} we may assume that $(v,\rho) \in C^2(\TT^d) \times C^2(\TT^d)$. Using $\rho$ as a test-function in \eqref{eq:mfg_lin_weak_1} and $v$ as a test-function in \eqref{eq:mfg_lin_weak_2} we get
  \begin{equation}\label{eq:proof_injective_1}
  \int D v \cdot D \rho + \rho Du \cdot Dv + \lambda v \rho \ dx = \int f'(m) \rho^2 \ dx,
 \end{equation}
 and
 \begin{equation}\label{eq:proof_injective_2}
  \int D \rho \cdot D v + \rho Du \cdot Dv + \lambda \rho v \ dx = \int -m \module{Dv}^2  \ dx.
 \end{equation}
 Subtracting \eqref{eq:proof_injective_2} from \eqref{eq:proof_injective_1} we obtain
 \begin{equation}\label{eq:existence_stable_solutions_1}
  \int m \module{Dv}^2 \ dx = - \int f'(m) \module{\rho}^2 \ dx.
 \end{equation}
 In the case where $f'\geq 0$, and since $m$ is nonnegative (see \cref{lem:fokker_planck}), this implies that $Dv = 0$ on the set where $m > 0$, in particular $mDv = 0$ on $\TT^d$. Using \cref{lem:fokker_planck}, we deduce that $\rho = 0$ on $\TT^d$. Using the uniqueness of the solution to the equation satisfied by $v$, we conclude that $v = 0$.

 We now assume that $\lambda > \Lambda$. Notice that this and \cref{prop:estimates_mfg} yield $\norm{m}{L^\infty} \leq 2\norm{m_0}{L^\infty}$. Then, using $\rho$ as a test function in \eqref{eq:mfg_lin_weak_2}, it follows from \eqref{eq:u_lipschitz}, the positivity of $m$ and \eqref{eq:existence_stable_solutions_1}, that
 \begin{align*}
  \int \module{D\rho}^2 & + \lambda \module{\rho}^2 \ dx \leq \norm{Du}{L^\infty} \norm{\rho}{L^2}\norm{D\rho}{L^2} + \norm{D \rho}{L^2} \norm{mDv}{L^2} \\
  & \quad \leq K \norm{\rho}{L^2}\norm{D\rho}{L^2} + \norm{m}{L^\infty}^{1/2}\norm{D \rho}{L^2} \norm{m^{1/2}Dv}{L^2} \\
  & \quad \leq K \norm{\rho}{L^2}\norm{D\rho}{L^2} + \norm{m}{L^\infty}^{1/2}\norm{D \rho}{L^2} \left(- \int f'(m)\module{\rho}^2 \right)^{1/2} \\
  & \quad \leq \norm{D\rho}{L^2}^2 + \frac{K^2}{2}\norm{\rho}{L^2}^2 + \frac{1}{2}\norm{f'}{L^\infty} \norm{m}{L^\infty} \norm{\rho}{L^2}^2 \\
  & \quad \leq \norm{D\rho}{L^2}^2 + \left (\frac{K^2}{2} + \norm{m_0}{L^\infty} \norm{f'}{L^\infty} \right) \norm{\rho}{L^2}^2.
 \end{align*}
 From our assumption on $\lambda$, we obtain $\norm{\rho}{L^2} = 0$ and the conclusion follows as in the first case.

\end{proof}

\section{Reformulation of the MFG system}
\label{section:Reformulation}
We are now going to reformulate \eqref{eq:mfg} as an abstract equation $F(u,m) = 0$.
In order to do this let us introduce an additional assumption which will be useful to state some results in a general form.
\begin{enumerate}[label=\bf{(H)}] 
\item \label{H:banach} $X$ and $Z$ are Banach spaces such that
\begin{equation}\label{eq:condition_banach}
  C^{2,\alpha}(\TT^d) \times C^{2,\alpha}(\TT^d)\subset X \subset H^1(\TT^d) \times L^2(\TT^d) \quad \textnormal{and} \quad Z \subset L^1(\TT^d) \times (W^{1,\infty}(\TT^d))'.
\end{equation}
with continuous embeddings. For every $(v,\rho)\in X$, we have that
\begin{equation}\label{eq:G}
 G(v,\rho):= \left ( \frac{1}{2} \module{Dv}^2 - f(\rho), - \diver \left (\rho Du \right ) - \lambda m_0 \right )\quad\text{belongs to } Z 
\end{equation}
and, for every $(\xi,\zeta)\in Z$, the equation
\begin{equation}\label{eq:T}
 \begin{cases}
  - \Delta v + \lambda v = \xi & \quad \textnormal{in } \TT^d, \\
  - \Delta \rho + \lambda \rho = \zeta  & \quad \textnormal{in } \TT^d,
 \end{cases}
\end{equation}
admits a unique distributional solution $T(\xi,\zeta)$ that belongs to $X$.
\end{enumerate}
Assumption \ref{H:banach} allows to define two mappings $G\colon X\to Z$ and $T\colon Z\to X$ such that, at least formally, $(u,m)$ solves~\eqref{eq:mfg} if and only 
\begin{equation}
\label{eq:mfg_abstract}
 F(u,m) : = (u,m) + T \left (G(u,m) \right ) = 0.
\end{equation}
This is made rigorous in the following result. 

\begin{prop}
Assume that $f \in W^{1,\infty}(\RR)$ and that \ref{H:banach} holds. A pair $(u,m) \in  C^{2,\alpha}(\TT^d) \times C^{2,\alpha}(\TT^d)$ solves \eqref{eq:mfg} if and only if it satisfies \eqref{eq:mfg_abstract}.
\end{prop}

\begin{proof}
 Assume that $(u,m) \in  C^{2,\alpha}(\TT^d) \times C^{2,\alpha}(\TT^d)$ is a solution to \eqref{eq:mfg}. We can rewrite \eqref{eq:mfg} as
 \[
   \begin{cases}
    - \Delta u + \lambda u = f(m) -  \frac{1}{2} \module{Du}^2 \quad & \textnormal{in } \TT^d, \\
    - \Delta m  + \lambda m = \lambda m_0 + \diver \left ( m Du \right ) \quad & \textnormal{in } \TT^d,
 \end{cases}
 \]
 or, with an obvious abuse of notation,
 \begin{equation}\label{eq:equivalence}
  (-\Delta + \lambda I) (u,m) = -G(u,m).
 \end{equation}
 We can apply $T = (-\Delta + \lambda I)^{-1}$ on both sides of \eqref{eq:equivalence} and use the linearity of $T$ to obtain that
 \[
  (u,m) = -T(G(u,m)).
 \]
 Conversely if $(u,m)$ satisfies \eqref{eq:mfg_abstract} we can apply the operator $(-\Delta + \lambda I)$ to get \eqref{eq:equivalence}.
\end{proof}

Let now $(u,m)$ be a classical solution to \eqref{eq:mfg} and $X$ and $Z$ be Banach spaces such that \ref{H:banach} holds and $(u,m) \in X$. Formally, the Fréchet differential of the mapping $F$ defined in \eqref{eq:mfg_abstract} at $(u,m)$ is given by
\begin{equation}\label{eq:diff_F}
 dF[u,m](v,\rho) := (v,\rho) + T \left (dG[u,m](v,\rho) \right)
\end{equation}
where
\begin{equation}\label{eq:dG}
 dG[u,m](v,\rho) = \left (Du \cdot Dv - f'(m) \rho, -\diver \left (\rho Du \right) - \diver \left (m Dv \right) \right ).
\end{equation}
The rigorous proof of Fréchet differentiability will be made in \cref{section:differentiability} below.
Notice that $(v,\rho)$ solves \eqref{eq:mfg_lin} if and only if
\begin{equation}
 dF[u,m](v,\rho) = 0.
\end{equation}

\begin{thm}[Isomorphism property of stable solutions]\label{thm:isom}
 Let $f \in C_b^1(\RR)$ and let $X$ and $Z$ be Banach spaces such that \ref{H:banach} holds and let $Y$ be a Banach space such that
 \[
  Y \subset X \subset H^1(\TT^d) \times L^2(\TT^d) \quad \textnormal{with} \quad Y \subset X \textnormal{ compact.}
 \]
 Assume also that the mapping $G \colon X \to Z$ defined by \eqref{eq:G} is continuously differentiable and that $T \in \cL(Z,Y)$. Then the mapping $F\colon X \to X$ defined by \eqref{eq:mfg_abstract} is continuously differentiable with $dF = I + T\circ dG$, and, for every stable solution $(u,m) \in C^{2,\alpha}(\TT^d) \times C^{2,\alpha}(\TT^d)$ to \eqref{eq:mfg}, the linear operator $dF[u,m]$ is an isomorphism on $X$.
\end{thm}

\begin{proof}
 The differentiability of $F$ is a direct consequence of the chain rule and the differentiability of $G$. If $(u,m)  \in C^{2,\alpha}(\TT^d) \times C^{2,\alpha}(\TT^d)$ is a stable solution to \eqref{eq:mfg}, then $(I + T \circ dG[u,m])$ is injective on $C^{2,\alpha}(\TT^d) \times C^{2,\alpha}(\TT^d)$. By \cref{lem:mfg_lin_reg} we know that any weak solution $(v,\rho) \in H^1(\TT^d) \times L^2(\TT^d)$ to \eqref{eq:mfg_lin} belongs to $C^{2,\alpha}(\TT^d) \times C^{2,\alpha}(\TT^d)$. Hence the operator $(I + T \circ dG[u,m])$ is also injective on $H^1(\TT^d) \times L^2(\TT^d)$, in particular it is injective on $X$.

 Notice now that since $dG[u,m] \in \cL(X,Z)$ and $T \in \cL(Z,Y)$ we have that $T \circ dG[u,m] \in \cL(X,Y)$ and since $Y \subset X$ is compact we deduce that the operator $T \circ dG[u,m] \in \cL(X)$ is compact. Using Fredholm's alternative, we conclude that $dF[u,m]$ is an isomorphism on $X$.
\end{proof}

We now provide two concrete examples of Banach spaces $X$, $Y$, and $Z$, satisfying \ref{H:banach} and the assumptions of \cref{thm:isom}. These examples will constitute the building blocks of the applications studied in \cref{section:applications}.

\begin{ex}\label{ex:holder}
 Let $f \in C^1_b(\RR)$, let $0 < \gamma < \beta < \alpha$ and set
 \[
  X = C^{2,\beta}(\TT^d) \times C^{2,\gamma}(\TT^d), \quad Y = C^{2,\alpha} (\TT^d)\times C^{2,\beta}(\TT^d) \quad \textnormal{and} \quad Z = C^{0,\alpha} (\TT^d)\times C^{0,\beta}(\TT^d)
 \]
 where $\alpha$ is fixed in \eqref{h:m0}.

 Let us check that \ref{H:banach} and the assumptions of \cref{thm:isom} are satisfied in this case. From the Arzela-Ascoli theorem we have the compact embedding of $Y$ into $X$. It is also clear that $X \subset H^1(\TT^d) \times L^2(\TT^d)$. We fix some stable solution $(u,m) \in C^{2,\alpha}(\TT^d) \times C^{2,\alpha}(\TT^d)$ to \eqref{eq:mfg}. From the Schauder estimates \cite[Corollary 6.3]{GT2001} we have $T \in \cL(Z,Y)$. It remains to prove that $G \colon X \to Z$ is well-defined. First notice that
 \[
  \norm{\module{Dv}^2}{C^{0,\alpha}} \leq 2 \norm{Dv}{C^{0,\alpha}}^2 \leq C \norm{v}{C^{2,\beta}}^2.
 \]
 Since $f$ is Lipschitz continuous we also have
 \[
  \norm{f(\rho)}{C^{0,\alpha}} \leq \norm{f}{W^{1,\infty}} \norm{\rho}{C^{0,\alpha}} \leq C \norm{f}{W^{1,\infty}} \norm{\rho}{C^{2,\gamma}}
 \]
 Finally, writting
 \[
  \diver(\rho Dv) = D\rho \cdot Dv + \rho \Delta v,
 \]
 we obtain that
 \[
  \norm{\diver(\rho Dv)}{C^{0, \beta}} \leq 2 \norm{D\rho}{C^{0,\beta}} \norm{Dv}{C^{0, \beta}} + 2 \norm{\rho}{C^{0, \beta}} \norm{\Delta v}{C^{0, \beta}} \leq C \norm{\rho}{C^{2,\gamma}}\norm{v}{C^{2,\beta}}.
 \]
 In conclusion we have
 \[
  \norm{G(v,\rho)}{Z} \leq C\left (\norm{v}{C^{2,\beta}}^2 + \norm{\rho}{C^{2,\gamma}}^2 \right) \leq C \norm{(v,\rho)}{X}^2,
 \]
 and the mapping $G$ is therefore well defined from $X$ to $Z$.

  In the case where $G$ is also continuously differentiable we have $dG \colon X \to \cL(X,Z)$ and we can apply \cref{thm:isom} to conclude that $dF[u,m]$ is an isomorphism on $C^{2,\beta}(\TT^d) \times C^{2,\gamma}(\TT^d)$.
\end{ex}

\begin{ex}\label{ex:sobolev}
    Let $f \in C^1_b(\RR)$ and
   \[
    X = W^{1,p}(\TT^d) \times L^q(\TT^d), \quad Y = W^{2,p/2}(\TT^d) \times W^{1,r}(\TT^d) \quad \textnormal{and} \quad Z = L^{p/2}(\TT^d) \times W^{-1, r'}(\TT^d)
   \]
   where $d < p,q < \infty$, $r, r' > 1$ are such that $1/r \geq 1/p + 1/q$, $1/r + 1/r' = 1$ and we have a compact embedding $W^{1,r}(\TT^d) \hookrightarrow L^q(\TT^d)$.

  We now verify that \ref{H:banach} and the assumptions in \cref{thm:isom} are satisfied. Since we assume $p > d$  the Rellich-Kondrachov theorem gives the compact embedding $W^{2,p/2}(\TT^d) \hookrightarrow W^{1,p}(\TT^d)$. It follows that there is a compact embedding from $Y$ into $X$.
   Using $W^{2,p}$ estimates \cite[Theorem 9.11]{GT2001} and the $W^{1,p}$ estimate \eqref{eq:W1p_estimate} from \cref{lem:fokker_planck}, and using \cref{rem:dual_sobolev}, we have that $T \in \cL(Z,W^{2,p/2}(\TT^d) \times W^{1,r}(\TT^d))$. Therefore $T \in \cL(Z,Y) \subset \cL(Z,X)$. Moreover is is easy to check that $G \colon X \to Z$ is well defined.

   Therefore in the case where $G$ is continuously differentiable we obtain from \cref{thm:isom} that $dF[u,m]$ is an isomorphism on $W^{1,p}(\TT^d) \times L^q(\TT^d)$.
\end{ex}

We now turn to the particular case of a monotone coupling and prove that the unique classical solution to \eqref{eq:mfg}, which we also know to be stable by \cref{thm:stability}, satisfies a stronger isomorphism property than the one resulting from the direct application of \cref{thm:isom}.

\begin{thm}[Isomorphism property for monotone couplings]\label{thm:mon_isom}
 Let $f \in C_b^{1}(\RR) \cap C^{1,1}_{\loc}(\RR)$. Assume that $f' \geq 0$ and let $(u,m) \in C^{2,\alpha}(\TT^d) \times C^{2,\alpha}(\TT^d)$, where $\alpha$ is fixed in \eqref{h:m0}, be the unique classical solution to \eqref{eq:mfg}. Then \ref{H:banach} holds for $X = C^{2,\alpha}(\TT^d) \times C^{2,\alpha}(\TT^d)$ and $Z = C^{0,\alpha}(\TT^d) \times C^{0,\alpha}(\TT^d)$.
 Moreover, assume that $G$ in \ref{H:banach} is continuously differentiable and define $F$ according to \eqref{eq:mfg_abstract}. Then $dF[u,m]$ is an isomorphism on $X$.
\end{thm}

\begin{proof}
According to \cref{thm:stability}, $(u,m)$ is a stable solution to \eqref{eq:mfg}. Therefore the injectivity of $dF[u,m]$ follows from the definition of stable solutions.

We now prove its surjectivity. More precisely we prove that for every $(w,\mu) \in  C^{2,\alpha}(\TT^d) \times C^{2,\alpha}(\TT^d)$ there exists $(v,\rho) \in  C^{2,\alpha}(\TT^d) \times C^{2,\alpha}(\TT^d)$ such that
\begin{equation}
 dF[u,m](v,\rho) = (w,\mu),
\end{equation}
which is equivalent to showing that there exists a unique solution $(v,\rho) \in  C^{2,\alpha}(\TT^d) \times C^{2,\alpha}(\TT^d)$ to
\begin{equation}
 \begin{cases}
  - \Delta v + Du \cdot Dv + \lambda v = f'(m) \rho - \Delta w + \lambda w &  \quad \textnormal{in } \TT^d, \\
  - \Delta \rho - \diver \left (\rho Du \right) + \lambda \rho = \diver \left (m Dv \right ) - \Delta \mu + \lambda \mu  &  \quad \textnormal{in } \TT^d.
 \end{cases}
\end{equation}

 The argument relies on the Leray-Schauder fixed point theorem and is adapted from \cite[Lemma 3.4]{CDLL2019}. Let us first define the mapping for which we will find a fixed point. Fix some $\rho \in L^2(\TT^d)$. Then there exists a solution $v \in H^2(\TT^d)$ to
 \[
  - \Delta v + Du \cdot Dv + \lambda v = f'(m) \rho - \Delta w + \lambda w \quad \textnormal{in } \TT^d,
 \]
 and a unique weak solution $\tilde \rho \in H^1(\TT^d)$ to
 \[
  - \Delta \tilde \rho - \diver \left (\tilde \rho Du \right) + \lambda \tilde \rho = \diver \left (m Dv \right) - \Delta \mu + \lambda \mu \quad \textnormal{in } \TT^d.
 \]
 This allows us to define a mapping $\Phi \colon L^2(\TT^d) \to L^2(\TT^d)$ by setting $\Phi(\rho) = \tilde \rho$. In order to apply the Leray-Schauder fixed point theorem \cite[Theorem 11.3]{GT2001} we have to prove that the set of solutions $\rho \in L^2(\TT^d)$ to $\rho = \sigma \Phi(\rho)$, where $\sigma \in [0,1]$, is uniformly bounded in $L^2(\TT^d)$. This amounts to prove a uniform bound on the solutions to
 \[
  \begin{cases}
   - \Delta v + Du \cdot Dv + \lambda v = \sigma \left [f'(m) \rho - \Delta w + \lambda w \right ] &  \quad \textnormal{in } \TT^d, \\
  - \Delta \rho - \diver \left (\rho Du \right) + \lambda \rho =\diver \left (m Dv \right ) - \sigma \left [ \Delta \mu - \lambda \mu \right ]  &  \quad \textnormal{in } \TT^d.
  \end{cases}
 \]
 Using $\rho$ as a test function for the equation satisfied by $v$ we obtain that
 \begin{equation}\label{eq:proof_isom_1}
  \int D v \cdot D \rho + \rho Du \cdot Dv + \lambda v \rho \ dx = \sigma \int f'(m) \rho^2 + \rho \left (\lambda w - \Delta w \right )\ dx.
 \end{equation}
 Similarly, using $v$ as a test function for the equation satisfied by $\rho$ we obtain that
 \begin{equation}\label{eq:proof_isom_2}
  \int D \rho \cdot D v + \rho Du \cdot Dv + \lambda \rho v \ dx = \int -m \module{Dv}^2 + \sigma \left (\lambda v \mu + Dv \cdot D\mu \right )\ dx.
 \end{equation}
 Subtracting \eqref{eq:proof_isom_2} to \eqref{eq:proof_isom_1} we get
 \[
  \int m \module{Dv}^2 \ dx = \sigma \int - f'(m) \rho^2 - \rho \left (\lambda w - \Delta w \right ) + \lambda v \mu +Dv \cdot D\mu \ dx.
 \]
 The positivity of $f'$ then yields
 \begin{equation}\label{eq:proof_isom_3}
  \int m \module{Dv}^2 \ dx \leq C \left (\norm{\rho}{L^2} \norm{w}{H^2} + \norm{v}{H^1} \norm{\mu}{H^1} \right ).
 \end{equation}

 Let now $\xi \in H^{-1}(\TT^d)$ and $z \in H^1(\TT^d)$ be the unique weak solution to
 \[
  - \Delta z + Du \cdot Dz  + \lambda z = \xi \quad \textnormal{in } \TT^d.
 \]
Using $z$ as a test function for the equation satisfied by $\rho$  and recalling that $m \geq 0$ (see \cref{lem:fokker_planck}-\ref{item:fp_positive}) we get
\begin{align*}
 \langle \xi, \rho \rangle_{H^{-1},H^1}  & = \int m Dv \cdot Dz + \sigma \left ( \lambda z \mu + D\mu \cdot Dz \right ) \ dx \\
 & \leq C  \left ( \norm{m^{1/2} Dz}{L^2} \norm{m^{1/2}Dv}{L^2} + \norm{z}{H^1} \norm{\mu}{H^1} \right ) \\
 & \leq C \norm{z}{H^1} \left ( \left ( \norm{\rho}{L^2} \norm{w}{H^2} + \norm{v}{H^1} \norm{\mu}{H^1} \right)^{1/2} + \norm{\mu}{H^1} \right ) \\
 & \leq C \norm{\xi}{H^{-1}} \left ( \left ( \norm{\rho}{L^2} \norm{w}{H^2} + \norm{v}{H^1} \norm{\mu}{H^1} \right)^{1/2} + \norm{\mu}{H^1} \right )
\end{align*}
where we used \eqref{eq:proof_isom_3} to obtain the second inequality. Since $\xi$ is arbitrary, we deduce by duality that
\[
 \norm{\rho}{H^1} \leq C \left ( \left ( \norm{\rho}{L^2} \norm{w}{H^2} + \norm{v}{H^1} \norm{\mu}{H^1}\right)^{1/2} + \norm{\mu}{H^1} \right ),
\]
and hence, by Young's inequality, we get
\[
 \norm{\rho}{H^1} \leq C \left ( \norm{w}{H^2} + \norm{v}{H^1}^{1/2}\norm{\mu}{H^1}^{1/2} + \norm{\mu}{H^1} \right ).
\]
From elliptic regularity we know that
\begin{align*}
 \norm{v}{H^2} & \leq C \left (\norm{\rho}{L^2} + \norm{w}{H^2} \right ) \leq C \left ( \norm{w}{H^2} + \norm{v}{H^1}^{1/2}\norm{\mu}{H^1}^{1/2} + \norm{\mu}{H^1} \right ),
\end{align*}
and hence, after another application of Young's inequality, we get a uniform bound on $v$ in $H^2(\TT^d)$
\[
 \norm{v}{H^2} \leq C \left ( \norm{w}{H^2} + \norm{\mu}{H^1} \right ).
\]

Using standard $H^1$ estimates we can now deduce a uniform on $\rho$ in $H^1(\TT^d)$
\[
 \norm{\rho}{H^1} \leq C \left (\norm{w}{H^2} + \norm{\mu}{H^1} \right ).
\]

The compactness on $\Phi$ then follows from the Rellich-Kondrachov theorem and we can therefore apply the Leray-Schauder fixed point theorem to obtain a pair $(v;\rho) \in H^2(\TT^d) \times H^1(\TT^d)$ solving \eqref{eq:mfg_lin}.

Using a bootstrap argument similar to the one used in \cref{lem:mfg_lin_reg}, we can obtain that $(v,\rho) \in C^{2,\alpha}(\TT^d) \times C^{2,\alpha}(\TT^d)$.

\end{proof}

\section{Differentiability of the mapping $F$}
\label{section:differentiability}

 In this section we prove the differentiability of the mapping $F$ defined by \eqref{eq:mfg_abstract} in both Hölder and Sobolev spaces. This corresponds to the situations considered in \cref{ex:holder,ex:sobolev}.

\subsection{Differentiability in Hölder spaces}

The following proposition summarizes the differentiability properties of the Nemytskii operator on Hölder spaces which we will use to prove the differentiability of the mapping $G$.
\begin{prop}\label{lem:f_diff_holder}
 Assume that $h \in C^2(\RR)$ and let $\beta \in (0,1]$. Then the Nemytskii operator $H \colon C^{0,\beta}(\TT^d) \to C^{0,\beta}(\TT^d)$ defined by
 \[
  H[u](x) = h(u(x)) \quad \textnormal{for } x \in \TT^d
 \]
 is continuously differentiable and $dH[u](v) = h'(u)v$. In particular, for every $\gamma \in (0,1]$ the mapping
 \[
  C^{2,\gamma}(\TT^d) \ni u \mapsto H[u] \in C^{0,\beta}(\TT^d)
 \]
 is also continuously differentiable. Furthermore, if $h \in C^{2,1}_{\loc}(\RR)$, then $dH$ is locally Lipschitz continuous in $\cL(C^{2,\gamma}(\TT^d), C^{0,\beta}(\TT^d))$.
\end{prop}
\begin{proof}
 The first statement is proved in \cite[Theorem 4.1]{N1993} (this is where we need the $C^2$ assumption on $f$). For the second one let us write $J \in \cL(C^{2,\gamma}(\TT^d), C^{0,\beta}(\TT^d))$ for the natural injection of $C^{2,\gamma}(\TT^d)$ into $C^{0,\beta}(\TT^d)$. It is enough to notice that $H \circ J$ is continuously differentiable from the chain rule.

 For the local Lipschitz continuity let $v \in C^{2,\gamma}(\TT^d)$. We have
 \[
  \norm{(h'(u_1) - h'(u_2))v}{C^{0,\beta}} \leq 2 \norm{v}{C^{0,\beta}} \norm{h'(u_1) - h'(u_2)}{C^{0,\beta}} \leq 2 \norm{v}{C^{2,\gamma}} \norm{h'(u_1) - h'(u_2)}{C^{0,\beta}}
 \]
 so that
 \[
  \norm{dH[u_1] - dH[u_2]}{\cL(C^{2,\gamma}, C^{0,\beta})} \leq  \norm{h'(u_1) - h'(u_2)}{C^{0,\beta}}.
 \]
 Therefore $dH$ is locally Lipschitz continuous as soon as
 \[
  C^{2,\gamma}(\TT^d) \ni u \mapsto h'(u) \in C^{0,\beta}(\TT^d)
 \]
 is locally Lipschitzian, which is the case if $h' \in C^{1,1}_{\loc}$ according to \cite[Theorem 3.1]{N1993}.
\end{proof}

We are now going to prove the differentiability of the mapping $F$, defined by \eqref{eq:mfg_abstract}, in the situations described in \cref{ex:holder} and \cref{thm:mon_isom}.

\begin{lem}\label{lem:diff_G_holder}
 Let $f \in C_b^{1}(\RR) \cap C^{2}(\RR)$ and $0 < \gamma \leq \beta \leq \alpha$, where $\alpha$ is set in \eqref{h:m0}. Set $X = C^{2,\beta}(\TT^d) \times C^{2,\gamma}(\TT^d)$ and $Z = C^{0,\alpha} (\TT^d)\times C^{0,\beta}(\TT^d)$. Then the mapping $G \colon X \to Z$, defined by \eqref{eq:G}, is continuously differentiable and, for every $(u,m) \in X$, we have
 \[
  dG[u,m](v,\rho) = \left ( Du \cdot Dv - f'(m) \rho, - \diver (\rho Du) - \diver (m Dv) \right).
 \]
 Furthermore, if $f \in C^{2,1}_{\loc}(\RR)$, then $dG$ is locally Lipschitz continuous.
\end{lem}

\begin{proof}
 From \cref{lem:f_diff_holder}, we have that the differentials of
 \[
  C^{2,\beta}(\TT^d) \ni u \mapsto \frac{\module{Du}^2}{2} \in C^{0,\beta}(\TT^d)
 \]
 and
 \[
  C^{2,\gamma}(\TT^d) \ni m \mapsto f(m) \in C^{0,\beta}(\TT^d),
 \]
 are given by
 \[
   C^{2,\beta}(\TT^d) \ni v \mapsto Du \cdot Dv \in C^{0,\beta}(\TT^d)
 \]
 and
 \[
  C^{2,\gamma}(\TT^d) \ni \rho \mapsto f'(m)\rho \in C^{0,\beta}(\TT^d),
 \]
respectively. The second component of $dG$ being continuous and bilinear, its differentiability is also easy to check. The remaining conclusions follow from the assumptions on $f$ and last assertion in \cref{lem:f_diff_holder}.
\end{proof}

Then fact that $F$ is continuously differentiable when $f \in C_b^{1}(\RR) \cap C^{2}(\RR)$ is then a direct consequence of \cref{lem:diff_G_holder} and the chain rule.

\begin{prop}\label{prop:F_diff_holder}
 Let $f \in C_b^{1}(\RR) \cap C^{2}(\RR)$ and let $0 < \gamma \leq \beta \leq \alpha$, where $\alpha$ is fixed in \eqref{h:m0}. Set $X = C^{2,\beta}(\TT^d) \times C^{2,\gamma}(\TT^d)$ and $Z = C^{0,\alpha} (\TT^d)\times C^{0,\beta}(\TT^d)$. Then the mapping $F \colon X \to X$ defined by \eqref{eq:mfg_abstract} is continuously differentiable with $dF$ given by \eqref{eq:diff_F}. Furthermore, if $f \in C^{2,1}_{\loc}(\RR)$, then $dF$ is locally Lipschitz continuous in $\cL(X)$.
\end{prop}

\subsection{Differentiability in Sobolev spaces}

We begin with a preliminary result on the differentiability of the Nemytskii operator on Lebesgue spaces.

\begin{prop}\label{prop:diff_nemytskii_lebesgue}
 Let $(\Omega, \mathcal{A}, \mu)$ be a finite measure space and let $1 < p,q,r < \infty$ with $q < p$ and $1/r = 1/q - 1/p$. Consider a function $h \in C^1(\RR)$ such that
 \begin{equation}\label{eq:nemytskii_growth}
  \module{h'(x)} \leq C \left ( 1 + \module{x}^{p/r} \right)
 \end{equation}
 and define the Nemytskii operator $H \colon L^p(\Omega, \mu) \to L^q(\Omega, \mu)$ by
 \[
  H[u](x) = h(u(x)) \quad \textnormal{for } x \in \Omega \quad \textnormal{for } \mu\textnormal{-a.e.} x \in \Omega.
 \]
 Then $H$ is continuously differentiable with $dH[u](v) = h'(u)v$.
\end{prop}

\begin{proof}
  The result is well-known and directly follows from \eqref{eq:nemytskii_growth} and \cite[Theorem 3.12]{AZ1990}.
\end{proof}

We now come back to the differentiability of the mapping $F$ for the situation described in \cref{ex:sobolev}.
\begin{prop}\label{prop:F_diff_sobolev}
 Let $f \in C_b^{1}(\RR) \cap C^{1,1}_{\loc}(\RR)$, $X = W^{1,p}(\TT^d) \times L^q(\TT^d)$, and $Z = L^{p/2}(\TT^d) \times W^{-1,r'}(\TT^d)$, where $d \leq p,q < \infty$, $q > p/2$, and $1/r \geq 1/p + 1/q$. Assume also that $W^{1,r}(\TT^d) \hookrightarrow L^{q}(\TT^d)$. Then the mapping $F \colon X \to X$ defined by \eqref{eq:mfg_abstract} is continuously differentiable.
 \end{prop}

 \begin{proof}
  Under these assumptions, we recall from \cref{ex:sobolev} that \ref{H:banach} holds. From \cref{prop:diff_nemytskii_lebesgue} we have that the mappings
  \[
    L^q(\TT^d) \ni m \mapsto f(m) \in L^{p/2}(\TT^d)
  \]
  and
  \[
    W^{1,p}(\TT^d) \ni u \mapsto \module{Du}^2 \in L^{p/2}(\TT^d)
  \]
  are continuously differentiable. Moreover
  \[
    W^{1,p}(\TT^d) \times L^q(\TT^d) \ni (u,m) \mapsto \diver(mDu) \in W^{-1,r}(\TT^d)
  \]
  is also continuously differentiable as a continuous bilinear operator. The differentiability of $F$ then follows from the chain rule.
 \end{proof}

\section{Applications}
\label{section:applications}
In this section we provide three applications of the isomorphism property of stable solutions to \eqref{eq:mfg}.

\subsection{Stability under perturbations of the MFG system}

In this section, our goal is to study perturbations of the mean field game system \eqref{eq:mfg}, or equivalently \eqref{eq:mfg_abstract}. In the case of stable solutions, the isomorphism property obtained in \cref{thm:isom,thm:mon_isom} motivates the use of the implicit function theorem, allowing us to consider a large class of perturbations. In what follows we provide a simple example of this idea. Namely, we consider the system
\begin{equation}\label{eq:perturbed_mfg}
 \begin{cases}
  - \Delta u + \frac{1}{2} \module{Du}^2 + \lambda u = f(m) + \epsilon \hat f(m) \quad & \textnormal{in } \TT^d, \\
  - \Delta m - \diver(mDu) + \lambda m = \lambda ((1 - \epsilon)m_0 + \epsilon m_1) \quad & \textnormal{in } \TT^d,
 \end{cases}
\end{equation}
where $\epsilon >0$ is a small parameter and $\hat f \in C^2((0,+ \infty))$ and $m_1 \in C^{0,\alpha}(\TT^d) \cap \cP(\TT^d)$ are perturbations of $f \in C_b^{1}(\RR) \cap C^{2}(\RR)$ and $m_0$, respectively. We recall that $\alpha$ is fixed in \eqref{h:m0}.

Let $(u,m) \in C^{2,\alpha}(\TT^d) \times C^{2,\alpha}(\TT^d)$ be a stable solution to \eqref{eq:mfg} and set $0 < \gamma < \beta < \alpha$. Since $m_0 \neq 0$ we also have $m > 0$ from  \cref{lem:fokker_planck} \ref{item:fp_positive}. In particular, there exists a  bounded neighborhood $\cO$ of $m$ in $C^{2,\gamma}(\TT^d)$ and $\eta > 0$ such that, for every $\tilde m \in \cO$, we have $ \tilde m \geq  \eta$. Notice that if we define
\[
 E_{\mathcal O} : = \left \{ \tilde m (x) : \, \tilde m \in \mathcal O, \, x \in \TT^d \right \},
\]
then we have that $E_{\mathcal{O}}$ is bounded in $\RR$ and $\inf E_{\mathcal O} \geq \eta > 0$.
Since $\hat f \in C^2((0,+\infty))$ is Lipschitz continuous on $E_{\mathcal O}$, we deduce from \cref{lem:f_diff_holder} that the mapping $\cO \ni \tilde m \mapsto \hat f(\tilde m) \in C^{0,\beta}(\TT^d)$ is well-defined and continuously differentiable.

Let $X = C^{2,\beta}(\TT^d) \times C^{2,\gamma}(\TT^d)$, $ Z = C^{0,\beta}(\TT^d) \times C^{0,\gamma}(\TT^d)$ and consider the mappings $G$ and $T$, defined by \eqref{eq:G} and \eqref{eq:T}, respectively.
We recall that it was checked in \cref{ex:holder} that \ref{H:banach} holds in this case. We introduce the mapping
\[
 \hat{G} \colon C^{2,\beta}(\TT^d) \times \cO \times \RR_+ \to C^{0,\beta}(\TT^d) \times C^{0,\gamma}(\TT^d)
\]
defined by
\[
 \hat{G}(u,m,\epsilon) = G(u,m) - \epsilon (\hat f(m), \lambda (m_1 - m_0) ).
\]
Then, setting
\[
 \hat{F}(u,m,\epsilon) = \left ( \hat I + T \circ \hat{G} \right)(u,m,\epsilon),
\]
where $T$ is defined by \eqref{eq:T} and $\hat I (u,m,\epsilon) = (u,m)$, it holds that $(u_\epsilon,m_\epsilon) \in X$ solves \eqref{eq:perturbed_mfg} if and only if
\[
 \hat{F}(u_\epsilon, m_\epsilon, \epsilon) = 0.
\]
Arguing as in the proof of \cref{prop:F_diff_holder}, we have that $\hat{F}$ is continuously differentiable on $C^{2,\beta}(\TT^d) \times \cO \times [0, +\infty)$. Moreover, it follows from \cref{thm:isom,ex:holder} that $d_{(u,m)}\hat{F}[u,m,0] = dF[u,m]$ is an isomorphism on $X$. We can apply the implicit function theorem to obtain the following result.

\begin{prop}[Sensitivity analysis]\label{prop:perturbations}
 Let $f \in C_b^{1}(\RR) \cap C^{2}(\RR)$, $\hat f \in C^{2}((0,+\infty))$, $m_1 \in C^{0,\alpha}(\TT^d) \cap \cP(\TT^d)$, $0 < \gamma < \beta < \alpha$, where $\alpha$ is fixed in \eqref{h:m0}, and $(u,m) \in  C^{2,\alpha}(\TT^d) \times C^{2,\alpha}(\TT^d)$ be a stable solution to \eqref{eq:mfg}. Then, for some $\epsilon_0 > 0$ and every $\epsilon \in [0, \epsilon_0)$, there exists $(u_\epsilon, m_\epsilon) \in C^{2,\beta}(\TT^d) \times C^{2,\gamma}(\TT^d)$  such that
 \[
  \hat{F}(u_\epsilon,m_\epsilon, \epsilon) = 0,
 \]
 with
 \[
 (u_\epsilon,m_\epsilon) = (u,m) - \epsilon dF(u,m)^{-1} T (\hat f(m), \lambda (m_1 - m_0) ) + o(\epsilon),
 \]
 and
 \[
  \norm{u_\epsilon - u}{C^{2,\beta}} + \norm{m_\epsilon - m}{C^{2,\gamma}}  = O(\epsilon).
 \]
 Furthermore, $(u_\epsilon,m_\epsilon)$ is a stable solution to \eqref{eq:perturbed_mfg}.
\end{prop}

\begin{rem}
 We may choose $\alpha = \beta = \gamma$ in \cref{prop:perturbations} under the additional assumption that $f' \geq 0$.
\end{rem}

\subsection{Finite Element approximation of the MFG system}

Our goal here is to obtain error estimates for the finite element approximation of a stable classical solution to \eqref{eq:mfg} by applying the following result of Brezzi-Rappaz-Raviart \cite{BRR1980} (see also \cite[Section IV.3]{GR1986}).
\begin{thm}[{\cite[Theorem 3.3 and Remark 3.5]{GR1986}}]\label{thm:Brezzi-al}
 Let $V,W$ be Banach spaces, let $T, T_h \in \cL(W,V)$, for every $h > 0$, and let $G \colon V \to W$ be a continuously differentiable mapping such that $dG$ locally Lipschitz continuous in $\cL(V,W)$. Set $F = I + T \circ G$ and let $x \in V$ be such that $F(x) = 0$ and $dF[x]$ is an isomorphism on $V$. If
 \begin{equation}\label{eq:BRR1980_assumption}
  \lim_{h \to 0} \norm{T - T_h}{\cL(W,V)} = 0,
 \end{equation}
 then there exists $h_0 > 0$ and a neighborhood $\mathcal{O}$ of $x$ in $V$ such that, for every $0 < h \leq h_0$, there exists $x_h \in V$ such that
 \begin{equation}\label{eq:FEM_equation_abstract}
  F_h(x_h) := (I + T_h \circ G)(x_h) = 0.
 \end{equation}
 Furthermore the following properties hold
 \begin{enumerate}[label=(\roman*)]
  \item $dF_h[x_h]$ is an isomorphism on $V$,
  \item $x_h \in \mathcal{O}$ for every $0 < h \leq h_0$ and there is no other solution to \eqref{eq:FEM_equation_abstract} in $\mathcal{O}$,
  \item there exists a constant $K > 0$, independent of $h$, such that
  \begin{equation}
   \norm{x - x_h}{V} \leq K \norm{(T - T_h)G(x)}{V}.
  \end{equation}
 \end{enumerate}
\end{thm}

We fix $d \leq 3$ and consider the following Banach spaces
\begin{equation}\label{eq:banach_spaces}
 X = W^{1,p}(\TT^d) \times L^q(\TT^d)  \quad \textnormal{and} \quad Z = L^{p/2}(\TT^d) \times H^{-1}(\TT^d),
\end{equation}
where $p,q \in (3,6)$ with $\frac{1}{p} + \frac{1}{q} = \frac{1}{2}$. Notice that under these assumptions we have $q > p/2$. Moreover the Rellich-Kondrachov theorem gives the compact embedding $H^1(\TT^d) \hookrightarrow L^{q}(\TT^d)$ since $q < 6$. Therefore all the assumptions in \cref{ex:sobolev} are satisfied (with $r = 2$). In particular, the linear operator $T \in \cL(Z,X)$, defined by \eqref{eq:T}, and the mapping $G \colon X \to Z$, given by \eqref{eq:G}, satisfy \ref{H:banach}.  We set $r := \frac{dp}{p + d} < \min \{p,d\}$ and we notice that Sobolev's inequality implies that $W^{2,r}(\TT^d) \hookrightarrow W^{1,p}(\TT^d)$.

For every $h >0$, let $\mathcal{T}_h$ be a quasi-uniform family of periodic triangulations of $[0,1]^d$ (see \cite[Definition 4.4.13]{BS2008}). Let also $V_h \subset W^{1,\infty}(\TT^d)$ be the associated finite element space induced by $\PP^1$ Lagrange finite elements. We define $S_h \in \cL(H^{-1}(\TT^d),H^1(\TT^d))$ by $S_h \xi = v_h$ where $v_h$ is the unique element in $V_h$ such that
\[
 \int_{\TT^d} Dv_h \cdot D\phi_h + \lambda v_h \phi_h \ dx = \langle \xi,\phi_h \rangle_{H^{-1},H^1} \quad \textnormal{for every } \phi_h \in V_h.
\]
In addition, we denote by $S \in \cL(H^{-1}(\TT^d), H^1(\TT^d))$ the linear operator defined by $S \xi = v$, where $v \in H^1(\TT^d)$ is the unique weak solution to
\begin{equation}\label{eq:FEM_continuous}
 -\Delta v + \lambda v = \xi \quad \textnormal{in } \TT^d.
\end{equation}
These linear operators are known to be well defined through the Lax-Milgram theorem and we have
\begin{equation}\label{eq:bound_lax_milgram}
 \norm{S}{\cL(H^{-1}, H^1)},\, \norm{S_h}{\cL(H^{-1},H^1)} \leq \frac{1}{\min \{1,\lambda \}}.
\end{equation}
We also have from \cite[Theorem 3.16, Theorem 3.18]{EG2004} that
\begin{equation}\label{eq:FEM_error_L2}
 \norm{(S - S_h) \xi}{L^2} \leq Ch \norm{(S - S_h) \xi}{H^1} \leq Ch \norm{S \xi}{H^1} \leq Ch\norm{\xi}{H^{-1}},
\end{equation}
where the constant $C$ is independent of $h$.

Let $\xi \in L^r(\TT^d)$. According to \cite[Theorem 3.21]{EG2004} and \cite[Theorem 8.5.3]{BS2008}, there exists $h_0 >0$ and a positive constant $C$ such that, for every $h \leq h_0$, it holds that
\begin{equation}\label{eq:FEM_W1p_stable}
 \norm{S_h \xi}{W^{1,p}} \leq C \norm{S \xi}{W^{1,p}}.
\end{equation}
Using the continuous embedding $W^{2,r}(\TT^d) \hookrightarrow W^{1,p}(\TT^d)$ and $W^{2,p}$ estimates \cite[Theorem 9.11]{GT2001}, we have
\begin{equation}\label{eq:FEM_W2p}
 \norm{S \xi}{W^{1,p}} \leq C \norm{S \xi}{W^{2,r}} \leq C \norm{\xi}{L^r}.
\end{equation}
Combining \eqref{eq:FEM_W1p_stable} and \eqref{eq:FEM_W2p}, for every $h \leq h_0$, we have $S_h \in \cL(L^r(\TT^d),W^{1,p}(\TT^d))$, where $\lVert S_h \rVert_ {\cL(L^r,W^{1,p})}$ is bounded from above by a constant which is independent of $h$. Moreover, up to the choice of a smaller $h_0$, from \cite[Theorem 3.21, Corollary 3.23]{EG2004} and \cite[Theorem 8.5.3]{BS2008} we also have, for every $h \leq h_0$, that $S_h \in \cL(L^p(\TT^d),W^{1,p}(\TT^d))$ with
\begin{equation}\label{FEM:error_W1p}
 \norm{(S - S_h) \xi}{W^{1,p}} \leq Ch \norm{S \xi}{W^{2,p}} \leq Ch \norm{\xi}{L^p},
\end{equation}
where the constant $C$ is independent of $h$. \\

We can now define the linear operator
\[
 T_h \in \cL(L^r(\TT^d) \times H^{-1}(\TT^d), W^{1,p}(\TT^d) \times H^{1}(\TT^d)),
\]
 with range in $X_h := V_h \times V_h$, by setting
\[
 T_h(\xi,\zeta) = (S_h \xi, S_h \zeta) \quad \textnormal{for every } (\xi,\zeta) \in L^r(\TT^d) \times H^{-1}(\TT^d).
\]
Notice that from \eqref{eq:T}, \eqref{eq:FEM_error_L2} and \eqref{eq:FEM_W1p_stable}, we also have
\[
 T_h \in \cL(L^p(\TT^d) \times H^{-1}(\TT^d), W^{1,p}(\TT^d) \times L^2(\TT^d))
\]
with
\begin{equation}\label{eq:FEM_error_Th}
 \norm{(T - T_h)(\xi,\zeta)}{W^{1,p} \times L^2} \leq Ch \norm{(\xi,\zeta)}{L^p \times H^{-1}}.
\end{equation}
In particular, using Sobolev inequalities (\cite[Theorem 9.9, Corollary 9.11, Theorem 9.12]{B2011}, we have obtained that
\[
 (T - T_h) \in \cL(L^r(\TT^d) \times H^{-1}(\TT^d), W^{1,p}(\TT^d) \times L^{s}(\TT^d))) \cap \cL(L^p(\TT^d) \times H^{-1}(\TT^d), W^{1,p}(\TT^d) \times L^2(\TT^d)),
\]
where $s$ is the critical Sobolev exponent for the continuous embedding $H^1(\TT^d) \hookrightarrow L^{s}(\TT^d)$, namely
\[
 \begin{cases}
  s = 6 \quad & \textnormal{if } d = 3, \\
  s \in [6, \infty) \quad & \textnormal{if } d = 2, \\
  s = \infty \quad & \textnormal{if } d =1.
 \end{cases}
\]

We set $\theta = \frac{p - d}{p}$ and $\theta^\star = \frac{(p-2)s - 2p}{(s - 2)p}$  so that
\[
 \frac{2}{p} = \frac{1 - \theta}{r} + \frac{\theta}{p}
\]
and
\[
 \frac{1}{q} = \frac{p-2}{2p} = \frac{1-\theta^\star}{s} + \frac{\theta^\star}{2}.
\]
 Using complex interpolation (see \cite[Chapter 2]{L2018}) we have that
\[
 L^{p/2}(\TT^d) = \left [L^r(\TT^d), L^p(\TT^d) \right]_{\theta} \quad \textnormal{and} \quad L^q(\TT^d) = \left [ L^{s}(\TT^d), L^2(\TT^d) \right]_{\theta^\star}.
\]
It follows from \cite[Theorem 2.6]{L2018} together with \eqref{eq:FEM_W1p_stable}, \eqref{eq:FEM_W2p}, \eqref{FEM:error_W1p}, \eqref{eq:bound_lax_milgram} and \eqref{eq:FEM_error_L2} that that
\[
 \norm{S-S_h}{\cL(L^{p/2},W^{1,p})} \leq \norm{S-S_h}{\cL(L^{r},W^{1,p})}^{1 - \theta} \norm{S - S_h}{\cL(L^p,W^{1,p})}^{\theta} \leq C h^{\theta}
\]
and
\[
 \norm{S-S_h}{\cL(H^{-1},L^q)} \leq \norm{S-S_h}{\cL(H^{-1},L^{s})}^{1 - \theta^\star} \norm{S - S_h}{\cL(H^{-1},L^2)}^{\theta^\star} \leq C h^{\theta^\star}.
\]
Noticing that $T - T_h = (S - S_h, S - S_h)$, we deduce that
\begin{align*}
 \norm{T - T_h}{\cL(L^{p/2} \times H^{-1}, W^{1,p} \times L^q)} & \leq \norm{S-S_h}{\cL(L^{p/2},W^{1,p})} + \norm{S-S_h}{\cL(H^{-1},L^q)} \\ &\leq C \left (h^\theta + h^{\theta^\star} \right).
\end{align*}

Since
\[
 \begin{cases}
  \theta^\star = (p-3)/p \quad & \textnormal{if } d =3, \\
  \theta^\star \in [(p-3)/p, (p-2)/p) \quad & \textnormal{if } d=2, \\
  \theta^{\star} = (p-2)/p \quad & \textnormal{if } d =1,
 \end{cases}
\]
we have $\sigma := \min \{ \theta, \theta^\star \} = \theta^\star$. Therefore we obtain that
\[
 T_h, (T-T_h) \in \cL(Z,X)
\]
with
\begin{align}
 \norm{T-T_h}{\cL(Z,X)} \leq C h^{\sigma} \label{eq:FEM_error_estimate},
\end{align}
for every $h \leq h_0 \leq 1$.

We can now apply \cref{thm:Brezzi-al} to obtain the following result.
\begin{thm}[Local convergence of finite element approximations]\label{thm:fem_approx}
 Let $f \in W^{2,\infty}(\RR)$ and let $(u,m)$ be a stable solution to \eqref{eq:mfg}. Let $X,Z$ be defined according to \eqref{eq:banach_spaces} with $3 <p,q<6$ and $1/2 = 1/p + 1/q$ and let $T_h$be defined as above. There exists $h_0 \in (0,1]$ and a neighborhood $\mathcal{O}$ of the origin in $X$ such that, for every $0 < h \leq h_0$,  there exists a solution $(u_h,m_h) \in X_h$ to
 \[
  F_h(u_h,m_h) := (u_h,m_h) + T_h(G(u_h,m_h)) = 0
 \]
 satisfying
 \begin{enumerate}[label=(\roman*)]
  \item $(u,m) - (u_h, m_h) \in \mathcal{O}$,

  \item $(u_h,m_h)$ is the unique solution to $F_h(u_h,m_h) = 0$ in $(u,m) + \mathcal{O}$,

  \item There exists a positive constant $C > 0$ such that
  \[
   \norm{(u-u_h,m-m_h)}{X} \leq Ch^\sigma,
  \]
  where
  \[
    \begin{cases}
      \sigma = (p-3)/p \quad & \textnormal{if } d =3, \\
      \sigma \in [(p-3)/p, (p-2)/p) \quad & \textnormal{if } d=2, \\
      \sigma = (p-2)/p \quad & \textnormal{if } d =1,
 \end{cases}
  \]
  \item \label{item:fem_approx_isom} $dF_h[u_h,m_h] \in \cL(X)$ is an isomorphism.

 \end{enumerate}
\end{thm}

\begin{proof}
 Using \cref{prop:F_diff_sobolev} we have that $F$ is continuously differentiable and from \cref{ex:sobolev} we know that $dF[u,m]$ is an isomorphism on $X$. Note that the assumption $f \in W^{2,\infty}(\RR)$ ensures that $dG$ is locally Lipschitz continuous in $\cL(X,Z)$.  Moreover, from \eqref{eq:FEM_error_estimate} we deduce that \eqref{eq:BRR1980_assumption} is satisfied. We can therefore apply \cref{thm:Brezzi-al} to obtain the conclusion.
\end{proof}

As a direct consequence of \cref{thm:fem_approx} \ref{item:fem_approx_isom} we deduce the local convergence of Newton's method for the discretized problem.

\begin{cor}[Local convergence of the discrete Newton method]
 Under the assumptions of \cref{thm:fem_approx}, let $(u_h,m_h) \in X_h$ be a solution to
 \begin{equation}\label{eq:mfg_discrete}
  F_h(u_h,m_h) = 0.
 \end{equation}
 Then there exists a neighborhood $\mathcal{O}$ of $(u_h,m_h)$ such that, if $(u_h^0,m_h^0) \in \mathcal{O}$, then the sequence $(u_h^k,m_h^k)$ given by Newton's method applied to \eqref{eq:mfg_discrete}, \emph{i.e.},
 \begin{equation}\label{eq:Newton_discrete}
    (u_h^{k+1},m_h^{k+1}) + T_h \left (G(u_h^k,m_h^k) + dG[u_h^k,m_h^k](u_h^{k+1}-u_h^k,m_h^{k+1} - m_h^k) \right ) = 0,
   \end{equation}
converges quadratically to $(u_h,m_h)$ in $X_h$.
\end{cor}

\begin{proof}
 Since $X_h \subset X$ is finite dimensional, and since $dF_h[u_h,m_h] \in \cL(X_h)$ is injective on $X_h$ by \cref{thm:fem_approx}-\ref{item:fem_approx_isom}, we have that $dF_h[u_h,m_h]$ is also an isomorphism on $X_h$. We conclude using standard results on Newton's method (see \cref{thm:newton_general} below).
\end{proof}

\begin{rem}
 Relation \eqref{eq:Newton_discrete} amounts to finding $(v_h, \rho_h) \in X_h$ such that
 \begin{align}
  \int D v_h \cdot D\phi + Du_h^k \cdot Dv_h \phi + \lambda v_h \phi \ dx & = \int f'(m_h^k)(\rho_h - m_h^k)\phi + \module{Du_h^k}^2 \phi \ dx, \label{eq:newton_discrete_1}\\
  \int D \rho_h \cdot D \psi + \rho_h Du_h^k \cdot D\psi + \lambda \rho_h \psi\ dx &  = \int m_h^k \left (Du_h^k - Dv_h \right ) \cdot D\psi \ dx, \label{eq:newton_discrete_2}
 \end{align}
for every $(\phi,\psi) \in X_h$.
\end{rem}

\subsection{Newton's method}

We recall here the convergence results Newton's method, see \cite[Corollary 2.1 p.120]{HPUU2009}, \cite[Proposition 5.1]{Z1993}, and \cite[Theorem 6E.2]{DR2014} for instance.

\begin{thm}[Classical Newton method]\label{thm:newton_general}
 Let $X$ and $Y$ be Banach spaces and $F \colon X \to Y$ be continuously differentiable. Let $\bar x \in X$ be such that $F(\bar x) = 0$ and $dF[\bar x] \in \cL(X,Y)$ is an isomorphism. Then there exists a neighborhood $\mathcal{O}$ of $\bar x$ in $X$ such that the sequence defined by
 \begin{equation}
  \begin{cases}
   x_0 \in \mathcal{O}, \\
   F(x_k) + dF[x_k](x_{k+1} - x_k) = 0,
  \end{cases}
 \end{equation}
 converges superlinearly to $\bar x$. Furthermore, if $dF$ is locally Lipschitz continuous in $\cL(X,Y)$, then the convergence is quadratic.
\end{thm}

Let $0 < \gamma < \beta < \alpha$, where $\alpha$ is given in \eqref{h:m0}, and set
\[
 X = C^{2,\beta}(\TT^d) \times C^{0,\gamma}(\TT^d) \quad \textnormal{and} \quad Z = C^{0,\alpha}(\TT^d) \times C^{0,\beta}(\TT^d)
\]
so that \ref{H:banach} is satisfied and consider the mapping $T,\,G$, and $F$ defined by \eqref{eq:T}, \eqref{eq:G}, and \eqref{eq:mfg_abstract}, respectively. From \cref{prop:F_diff_holder} we know that the mapping $F$ is continuously differentiable with $dF$ given by \eqref{eq:diff_F}. Moreover, if we fix a stable solution $(u,m) \in C^{2,\alpha}(\TT^d) \times C^{2,\alpha}(\TT^d)$, then we know from \cref{ex:holder} that $dF[u,m]$ is an isomorphism on $X$.

A direct application of \cref{thm:newton_general} yields the following theorem.

\begin{thm}\label{thm:newton_1}
   Consider $f \in C_b^{1}(\RR) \cap C^{2}(\RR)$, let $0 < \gamma < \beta < \alpha$ and  $(u,m)\in C^{2,\alpha}(\TT^d) \times C^{2,\alpha}(\TT^d)$ be a stable solution to \eqref{eq:mfg}. Then there exists a neighborhood $\mathcal{O}$ of $(u,m)$ in $C^{2,\beta}(\TT^d) \times C^{2,\gamma}(\TT^d)$ such that, if $(u_0,m_0) \in \mathcal{O}$, then the sequence $(u_k,m_k)$ generated by Newton's method applied to \eqref{eq:mfg_abstract}, \textit{i.e.},
   \begin{equation}\label{eq:Newton_continuous}
    (u_{k+1},m_{k+1}) + T \left (G(u_k,m_k) + dG[u_k,m_k](u_{k+1}-u_k,m_{k+1} - m_k) \right ) = 0,
   \end{equation}
   converges super-linearly to $(u,m)$ in $C^{2,\beta}(\TT^d) \times C^{2,\gamma}(\TT^d)$. Furthermore, if we also assume $f \in C_{\loc}^{2,1}(\RR)$, then the convergence is quadratic.
\end{thm}

\begin{rem}
 At each iteration, the relation \eqref{eq:Newton_continuous} amounts to solving the linear system
 \begin{equation}\label{eq:Newton_continuous_pde}
  \begin{cases}
   - \Delta v + Du_k \cdot Dv + \lambda v = f'(m_k)(\rho - m_k) - \module{Du_k}^2 \quad & \textnormal{in } \TT^d, \\
   - \Delta \rho - \diver(\rho Du_k) + \lambda \rho = \diver(m_k Dv) - \diver(m_k Du_k) \quad & \textnormal{in } \TT^d.
  \end{cases}
 \end{equation}
\end{rem}

In the case where $f'\geq 0$, we may use \cref{thm:mon_isom} instead of \cref{ex:holder} to obtain a slightly better result than \cref{thm:newton_1}.

\begin{thm}\label{thm:newton_2}
   Let $f \in C_b^{1}(\RR) \cap C^{2}(\RR)$ with $f' \geq 0$. Let $(u,m)\in C^{2,\alpha}(\TT^d) \times C^{2,\alpha}(\TT^d)$ be a stable solution to \eqref{eq:mfg}. Then there exists a neighborhood $\mathcal{O}$ of of $(u,m)$ in $C^{2,\alpha}(\TT^d) \times C^{2,\alpha}(\TT^d)$ such that if $(u_0,m_0) \in \mathcal{O}$ then the sequence $(u_k,m_k)$ generated by Newton's method \eqref{eq:Newton_continuous} converges super-linearly to $(u,m)$ in $C^{2,\alpha}(\TT^d) \times C^{2,\alpha}(\TT^d)$. Furthermore, if we also assume $f \in C_{\loc}^{2,1}(\RR)$, then the convergence is quadratic.
\end{thm}

Finally we may also set $X = W^{1,p}(\TT^d) \times L^q(\TT^d)$ and $Z = L^{p/2}(\TT^d) \times W^{-1,r'}(\TT^d)$, where $d < p,q < \infty$ and $r > 1$ is such that $1/r = 1/p + 1/q$ and large enough so that there is a compact embedding $W^{1,r}(\TT^d) \hookrightarrow L^q(\TT^d)$. Then, we may replace \cref{prop:F_diff_holder} and \cref{ex:holder} by \cref{prop:F_diff_sobolev} and \cref{ex:sobolev}, respectively, in the discussion above to obtain the convergence of Newton's method in Sobolev spaces. The point being that in this case the neighborhood for the initial guess is expected to be less restrictive.

\begin{thm}\label{thm:newton_3}
   Let $f \in C_b^{1}(\RR) \cap C^{1,1}_{\loc}(\RR)$, let $X$ and $Z$ be as above and $(u,m)\in C^{2,\alpha}(\TT^d) \times C^{2,\alpha}(\TT^d)$ be a stable solution to \eqref{eq:mfg}. Then there exists a neighborhood $\mathcal{O}$ of $(u,m)$ in $ W^{1,p}(\TT^d) \times L^q(\TT^d)$ such that, if $(u_0,m_0) \in \mathcal{O}$, then the sequence $(u_k,m_k)$ generated by Newton's method \eqref{eq:Newton_continuous} converges superlinearly to $(u,m)$ in $ W^{1,p}(\TT^d) \times L^q(\TT^d)$. Furthermore, if $f \in W^{2,\infty}(\RR)$, then the convergence is quadratic.
\end{thm}

\appendices

\section{Proof of \cref{thm:well_posed}}
\label{section:well_posed}

 We are going to apply Schauder's fixed point theorem in $C^{0,\beta}(\TT^d)$ for some $\beta \in (0,\alpha]$ to be determined.

 Fix some $m \in C^{0,\beta}(\TT^d)$. Since $f$ is assumed to be Lipschitz continuous, we have that $f(m) \in C^{0,\beta}(\TT^d)$. From the standard theory of elliptic equation (see \cite[Theorem 15.12]{GT2001}, using the gradient bound in \cref{prop:estimates_mfg}, for the result with Dirichlet boundary conditions) we know that there exists a unique classical solution $u \in C^{2,\beta}(\TT^d)$ to
 \begin{equation}\label{eq:well_posed_1}
  - \Delta v + \frac{1}{2} \module{Dv}^2 + \lambda v = f(m) \quad \textnormal{in } \TT^d.
 \end{equation}
 Then from standard Schauder theory (see \cite[Corollary 6.3]{GT2001}) and \cref{lem:fokker_planck}, we also have a unique solution $\tilde m \in C^{2,\beta}(\TT^d)$ to
 \begin{equation}
   - \Delta \rho - \diver \left ( \rho Du \right ) + \lambda \rho = \lambda m_0 \quad \textnormal{in } \TT^d.
 \end{equation}
 This defines a mapping $\Phi \colon C^{0,\beta}(\TT^d) \to C^{0,\beta}(\TT^d)$ by setting $\Phi(m) = \tilde m$.

 We now prove that $\Phi$ is continuous. Let $(m_n)_{n \geq 0}$ be a sequence in $C^{0,\beta}(\TT^d)$ converging to some $m$ in $C^{0,\beta}(\TT^d)$. In particular, this sequence is bounded in $C^{0,\beta}(\TT^d)$. Using the fact that $f$ is Lipschitz continuous it follows that $f(m_n)$ is also bounded in $C^{0,\beta}(\TT^d)$. From the inequality \eqref{eq:u_lipschitz} in \cref{prop:estimates_mfg}, we have the existence of a positive constant $K$ such any solution classical solution $u_n$ to \eqref{eq:well_posed_1}, with $m$ replaced by $m_n$, satisfies
 \[
  \norm{Du_n}{L^\infty} \leq K
 \]
 and the constant $K$ depends on the right-hand side of \eqref{eq:well_posed_1} only through $\norm{f}{L^\infty}$ (and hence is independent of $n$). Then, from \cite[Theorem 13.6]{GT2001} we deduce that there exist constants $\gamma \in (0,1)$ and $\tilde K >0$, depending on $K$ and independent of $n$, such that
 \[
  \norm{Du_n}{C^{0,\gamma}} \leq \tilde K.
 \]
 Using Schauder estimates \cite[Corollary 6.3]{GT2001}, one has that the sequence $u_n$ of solutions to \eqref{eq:well_posed_1} associated to $m_n$ is bounded in $C^{2,\beta}(\TT^d)$ for $\beta = \min \{ \alpha, \gamma \}$. From the Arzela-Ascoli theorem, it admits a subsequence converging in $C^2(\TT^d)$ to a solution $u$ to \eqref{eq:well_posed_1} associated to $m$. Since this solution is unique the whole sequence must converge to this limit $u$. Then, using again Schauder estimates we also have that $\Phi(m_n)$ is bounded in $C^{2,\beta}(\TT^d)$ by a constant depending on $\tilde K$, and a similar argument shows that it must converge to $\Phi(m)$.

 We now claim the $\Phi(C^{0,\beta}(\TT^d))$ is bounded in $C^{0,\gamma}(\TT^d)$ for $\gamma \in (\beta,1)$. Indeed let us choose $1 < p < \infty$, depending only on $d$ and $\gamma$, such that $W^{1,p}(\TT^d) \hookrightarrow C^{0,\beta}(\TT^d)$. From the $W^{1,p}$ estimates \eqref{eq:W1p_estimate} in \cref{lem:fokker_planck}, we have that $\Phi(C^{0,\beta}(\TT^d))$ is bounded in $W^{1,p}(\TT^d)$ and therefore also in $C^{0,\gamma}(\TT^d)$. Using the Arzela-Ascoli theorem we deduce that $\Phi(C^{0,\beta}(\TT^d))$ is compact in $C^{0,\beta}(\TT^d)$.

 We can now apply Schauder's fixed point theorem \cite[Corollary 11.2]{GT2001} to obtain a classical solution to \eqref{eq:mfg}.

 The argument for uniqueness under the assumption that $f' \geq 0$ is a straightforward adaptation of the one introduced in \cite{LL2007}. We therefore turn to proof of uniqueness for large values of $\lambda$.
 
 Let $(u_1,m_1)$ and $(u_2, m_2)$ be two (classical) solutions to \eqref{eq:mfg} and set $v := u_1 - u_2$ and $\rho := m_1 - m_2$. Recall from \cref{prop:estimates_mfg} that there are positive constants $K$ and $M$, independent of $\lambda$, such that
 \[
  \max \left \{ \norm{D u_1}{L^\infty} , \norm{Du_2}{L^\infty} \right \} \leq K \quad \textnormal{and} \quad \max \left \{ \norm{m_1}{L^\infty} , \norm{m_2}{L^\infty} \right \} \leq 2 \norm{m_0}{L^\infty},
 \]
 if $\lambda > 2M$.
 Then, the pair $(v,\rho)$ satisfies 
 \begin{equation}\label{eq:uniqueness_lambda}
  \begin{cases}
   - \Delta v + \left ( D u_1 + Du_2 \right) \cdot Dv + \lambda v = f(m_1) - f(m_2) \quad & \textnormal{in } \TT^d, \\
   - \Delta \rho - \diver \left ( \rho Du_1 \right) + \lambda m = \diver \left (m_2 D v \right) \quad & \textnormal{in } \TT^d.
  \end{cases}
 \end{equation}
 Using $\rho$ as a test-function in the second equation in \eqref{eq:uniqueness_lambda} we get 
 \begin{align*}
  \int_{\TT^d} \module{D \rho}^2 + \lambda \module{\rho}^2 \, dx & \leq - \int_{\TT^d} \rho Du_1 \cdot D \rho + m_2 Dv \cdot D \rho \, dx \\
  & \leq \norm{D u_1}{L^\infty} \norm{\rho}{L^2} \norm{D\rho}{L^2} + \norm{m_2}{L^\infty} \norm{Dv}{L^2} \norm{D\rho}{L^2} \\
  & \leq K \norm{\rho}{L^2} \norm{D\rho}{L^2} + 2 \norm{m_0}{L^\infty} \norm{Dv}{L^2} \norm{D \rho}{L^2} \\
  & \leq \norm{D \rho}{L^2}^2 + \frac{K^2}{2} \norm{\rho}{L^2}^2 + 2 \norm{m_0}{L^\infty}^2 \norm{D v}{L^2}^2.
 \end{align*}
 Assuming, for instance, that $\lambda \geq 1 + \frac{1}{2} K^2$, we deduce that
 \[
  \norm{\rho}{L^2}^2 \leq 2 \norm{m_0}{L^\infty}^2 \norm{Dv}{L^2}^2.
 \]
 We then use $v$ as a test-function in the first equation in \eqref{eq:uniqueness_lambda} to obtain 
 \begin{align*}
  \int_{\TT^d} \module{Dv}^2 + \lambda \module{v}^2 \, dx & = \int_{\TT^d} \left ( f(m_1) - f(m_2) \right) v - v \left ( Du_1 + Du _2 \right) \cdot Dv \, dx \\
  & \leq \norm{f(m_1) - f(m_2)}{L^2} \norm{v}{L^2} + \norm{Du_1 + Du_2}{L^\infty} \norm{v}{L^2} \norm{Dv}{L^2} \\
  & \leq \norm{f'}{L^\infty} \norm{\rho}{L^2} \norm{v}{L^2} + 2K \norm{v}{L^2} \norm{Dv}{L^2} \\
  & \leq \left ( \sqrt{2} \norm{m_0}{L^\infty} \norm{f'}{L^\infty} + 2K \right) \norm{v}{L^2} \norm{Dv}{L^2} \\
  & \leq \norm{Dv}{L^2}^2 + \frac{\left ( \sqrt{2} \norm{m_0}{L^\infty} \norm{f'}{L^\infty} + 2K \right)^2}{4} \norm{v}{L^2}^2.
 \end{align*}
 It follows that $v = 0$ if $\lambda > \frac{\left ( \sqrt{2} \norm{m_0}{L^\infty} \norm{f'}{L^\infty} + 2K \right)^2}{4}$ and then, using \cref{lem:fokker_planck}-\ref{item:fp_uniqueness} for the second equation in \eqref{eq:uniqueness_lambda}, that $\rho =0$. This concludes the proof.
 
 \section{Proof of \cref{prop:isolated}}
 \label{section:isolated}
  The argument is adapted from \cite[Proposition 4.2]{BC2018}.  Assume that the conclusion is false, then there exists a sequence $(u_n,m_n)$ of classical solutions to \eqref{eq:mfg} converging to $(u,m)$ in $H^1(\TT^d) \times L^2(\TT^d)$. Note that we may assume that the convergence also holds in the almost everywhere sense. We then set
 \[
  \delta_n = \norm{(u_n,m_n) - (u,m)}{H^1 \times L^2}
 \]
 and
 \[
  v_n = \delta_n^{-1}(u_n - u), \quad \rho_n = \delta_n^{-1}(m_n - m),
 \]
 so that $\norm{(v_n, \rho_n)}{H^1 \times L^2} = 1$ for all $n \in \NN$.
 Then, for every $n$, the pair $(v_n,\rho_n)$ is a classical solution to
 \begin{equation}\label{eq:isolated_1}
  \begin{cases}
   -\Delta v_n + \lambda v_n = g_n \quad & \textnormal{in } \TT^d, \\
   -\Delta \rho_n + \lambda \rho_n = \diver(h_n) \quad & \textnormal{in } \TT^d,
  \end{cases}
 \end{equation}
 where
 \[
  g_n = \delta_n^{-1} \left ( f(m_n) - f(m) + \frac{1}{2} \module{Du}^2 - \frac{1}{2} \module{Du_n}^2 \right )
 \]
 and
 \[
  h_n = \delta_n^{-1} \left ( m_n Du_n - m Du \right ).
 \]
 We then notice that, for every $p \geq 1$,
 \begin{equation}\label{eq:isolated_gn}
 \begin{split}
  \norm{g_n}{L^p} & \leq \delta_n^{-1} \left ( \norm{f'}{L^\infty} \norm{m_n - m}{L^p} + K \norm{Du_n - Du}{L^p} \right ) \\
  & = \norm{f'}{L^\infty} \norm{\rho_n}{L^p} + K \norm{Dv_n}{L^p}
  \end{split}
 \end{equation}
 and
 \begin{equation}\label{eq:isolated_hn}
 \begin{split}
  \norm{h_n}{L^p} & \leq \delta_n^{-1} \left ( K \norm{m_n - m}{L^p} + \norm{m}{L^\infty} \norm{Du_n - Du}{L^p} \right) \\
  & =  K \norm{\rho_n}{L^p} + C \norm{Dv_n}{L^p},
  \end{split}
 \end{equation}
 where $K$ is given in \cref{prop:estimates_mfg} and $C$ is a constant independent of $n$, given by using \cref{lem:fokker_planck}-\ref{item:fp_DGNM}. Since $\norm{(v_n,\rho_n)}{H^1 \times L^2} = 1$ for every $n\geq 1$, we first obtain that there exists a constant $C_1 > 0$ such that
  \[
   \norm{h_n}{L^2} + \norm{g_n}{L^2} \leq C_1 \quad \textnormal{for all } n \geq 1.
  \]
  Using \eqref{eq:isolated_1} and elliptic regularity we deduce that there exists $\tilde C_1 > 0$ such that
  \[
   \norm{v_n}{H^2} + \norm{\rho_n}{H^1} \leq C_1 \quad \textnormal{for all } n \geq 1.
  \]
  Using \eqref{eq:isolated_gn}, \eqref{eq:isolated_hn} and Sobolev's inequality, there exists $C_2 > 0$ such that
  \[
   \norm{h_n}{L^{2d/(d-2)}} + \norm{g_n}{L^{2d/(d-2)}} \leq C_1 \quad \textnormal{for all } n \geq 1
  \]
  if $d > 2$ and
  \[
   \norm{h_n}{L^p} + \norm{g_n}{L^p}\leq C_1 \quad \textnormal{for all } n \geq 1
  \]
  otherwise. Using a bootstrap argument and Morrey's inequality we conclude that there exists $\bar C > 0$ and $\beta \in (0,1)$ such that
  \[
   \norm{v_n}{C^{1,\beta}} + \norm{\rho_n}{C^{0,\beta}} \leq \bar C \quad \textnormal{for all } n \geq 1.
  \]
  We may therefore extract a subsequence, still denoted by $(v_n, \rho_n)$, such that there exists $(v,\rho) \in C^1(\TT^d) \times C(\TT^d)$ so that
  \begin{equation}\label{eq:isolated_compactness}
   \begin{cases}
    v_n \to v \quad & \textnormal{in } C^1(\TT^d), \\
    \rho_n \to \rho \quad & \textnormal{in } C(\TT^d).
   \end{cases}
  \end{equation}
  Notice that Lebesgue's convergence theorem then implies that
  \begin{equation}\label{eq:isolated_norm_limit}
   \norm{(v,\rho)}{H^1 \times L^2} = \lim_{n \to \infty} \norm{(v_n,\rho_n)}{H^1 \times L^2} = 1.
  \end{equation}
 We now claim that the following facts hold
 \begin{equation}\label{eq:isolated_2}
  \begin{cases}
   \delta_n^{-1}(f(m_n) - f(m)) \xrightarrow{} f'(m)\rho \quad & \textnormal{in } L^2(\TT^d), \\
   (2\delta_n)^{-1}\left (\module{Du}^2 - \module{Du_n}^2 \right) \xrightarrow{} Du \cdot Dv \quad & \textnormal{in } L^2(\TT^d), \\
   \delta_n^{-1}(m_n Du_n - mDu) \xrightarrow{} \rho Du + m Dv \quad & \textnormal{in } L^{2}(\TT^d; \RR^d).
  \end{cases}
 \end{equation}
 Let us prove the first line of \eqref{eq:isolated_2}. Since $f \in C^1_b(\RR)$ we have that
 \begin{equation}\label{eq:isolated_claim_1}
  \delta_n^{-1}(f(m_n) - f(m)) = \rho_n \int_0^1 f'(\lambda m_n + (1-\lambda)m)\ d\lambda \quad \textnormal{in } \TT^d.
 \end{equation}
 We then decompose
 \begin{align}
  & \norm{\int_0^1 f'(\lambda m_n + (1-\lambda)m) \rho_n - f'(m)\rho \ d\lambda}{L^2} \nonumber \\ & \qquad \leq \norm{\rho_n \int_0^1 f'(\lambda m_n + (1-\lambda)m) - f'(m) \ d\lambda}{L^2} + \norm{f'(m)(\rho_n - \rho)}{L^2}. \label{eq:isolated_claim_2}
 \end{align}
 Since we assume that $f'$ is bounded, the convergence to zero of the terms in the right-hand side of the last inequality follows from Lebesgue's convergence theorem for the first one, and from the convergence of $\rho_n$ to $\rho$ in $L^2(\TT^d)$ for the second one. This proves the first line in \eqref{eq:isolated_2}.

For the second line of \eqref{eq:isolated_2}, we write
 \begin{align*}
  & \norm{(2\delta_n)^{-1}\left (\module{Du}^2 - \module{Du_n}^2 \right) - Du \cdot Dv}{L^2} = \frac{1}{2} \norm{(Du + Du_n)\cdot Dv_n - 2Du\cdot Dv}{L^2} \\
  & \qquad \leq \frac{1}{2} \norm{Du \cdot(Dv_n - Dv)}{L^2} + \frac{1}{2}\norm{Du_n \cdot Dv_n - Du \cdot Dv}{L^2} \\
  & \qquad \leq \frac{1}{2} \norm{Du}{L^\infty} \norm{Dv_n - Dv}{L^2} + \frac{1}{2}\norm{Du_n \cdot (Dv_n - Dv)}{L^2} + \frac{1}{2} \norm{ (Du_n - Du) \cdot Dv}{L^2} \\
  & \qquad \leq K \norm{Dv_n - Dv}{L^2}  + \frac{1}{2} \norm{ (Du_n - Du) \cdot Dv}{L^2} \\
  & \qquad \leq K \norm{Dv_n - Dv}{L^2}  + \frac{\bar C}{2} \norm{Du_n - Du}{L^2},
 \end{align*}
 The conclusion then follows from \eqref{eq:isolated_compactness} and the convergence of $u_n$ to $u$ in $H^1(\TT^d)$.

 For the third line of \eqref{eq:isolated_2}, we first notice that
   \[
    h_n = \delta_n^{-1}(m_n Du_n - mDu) = m_n Dv_n + \rho_n Du.
   \]
   It follows that
   \begin{align*}
    \norm{h_n - \rho Du + m Dv}{L^2} & \leq \norm{m_n(Dv_n - Dv)}{L^2} + \norm{(m_n - m)Dv}{L^2} + \norm{(\rho_n - \rho) Du}{L^2} \\
    & \leq \norm{m_n}{L^\infty} \norm{Dv_n - Dv}{L^2} + \norm{Dv}{L^\infty} \norm{m_n - m}{L^2} + K \norm{\rho_n - \rho}{L^2}
   \end{align*}
   and the conclusion follows from \eqref{eq:isolated_compactness} and the fact that the sequence $(m_n)_{n \in \NN}$ converges to $m$ in $L^2(\TT^d)$ and is bounded in $L^\infty(\TT^d)$ according to \cref{lem:fokker_planck}-\ref{item:fp_DGNM}.

   We can now pass to the limit in the weak formulation of \eqref{eq:isolated_1} to obtain that $(v,\rho) \in H^1(\TT^d) \times L^2(\TT^d)$ is a weak solution to \eqref{eq:mfg_lin}.
   According to \cref{lem:mfg_lin_reg}, $(v,\rho)$ is a classical solution and, since $(u,m)$ is assumed to be stable, we must have $(v,\rho) = (0,0)$. This contradicts the fact that $\norm{(v,\rho)}{H^1 \times L^2} = 1$ and concludes the proof.

\paragraph{Acknowledgement.} The authors would like to thank Y. A. P. Osborne and I. Smears (University College London) for pointing out an error in a preliminary version of this paper. In addition, we are also grateful to the anonymous referee for very relevent comments that helped us to improve the paper.

This work was partially supported by the ANR (Agence Nationale de la Recherche) through the COSS project ANR-22-CE40-0010 and the Centre Henri Lebesgue ANR-11-LABX-0020-01. The third author was partially supported by KAUST through the subaward agreement ORA-2021-CRG10-4674.6.

\printbibliography

\end{document}